\documentclass[11pt,draft,reqno]{amsart}
\usepackage{amsmath,amsfonts,amssymb}
\usepackage[utf8]{inputenc}
\usepackage{mathrsfs}

\def\wt{\widetilde}
\def\R{\mathbb R}
\def\C{\mathbb C}
\def\Z{\mathbb Z}

\def\N{\mathbb N}
\def\A{\mathbb A}
\def\B{\mathbb B}

\def\x{\mathbf x}
\def\y{\mathbf y}
\def\n{\mathbf n}
\def\m{\mathbf m}

\def\k{\mathbf k}

\def\z{\mathbf z}

\def\D{\mathbf D}

\def\1{\mathbf 1}

\def\H{\mathfrak H}

\def\bxi{\boldsymbol \xi}

\def\1{\bold 1}

\def\eps{\varepsilon}

\def\le{\leqslant}

\def\ge{\geqslant}

\usepackage[height = 23.5cm, a4paper, hmargin={2.5cm , 2.0cm}]{geometry}

\usepackage{color}

\definecolor{darkred}{rgb}{0.9,0.1,0.1}

\theoremstyle{theorem}
\newtheorem{theorem}{Theorem}[section]
\newtheorem{proposition}[theorem]{Proposition}
\newtheorem{lemma}[theorem]{Lemma}

\newtheorem{remark}[theorem]{Remark}
\newtheorem{corollary}[theorem]{Corollary}

\numberwithin{equation}{section}

\theoremstyle{plain}
\newtoks\thehProclaim
\newtheorem*{Proclaim}{\the\thehProclaim}

\begin{document}
\thispagestyle{empty}

\centerline{\textbf{Homogenization of non-symmetric convolution type operators}}

\bigskip
\def\supind#1{${}^\mathrm{#1}$}
\centerline{\textbf{} A.~L.~Piatnitski\supind{1,2},
V.~A.~Sloushch\supind{3}, T.~A.~Suslina\supind{3}, E.~A.~Zhizhina\supind{1,2}}
\bigskip
\centerline{\supind{1} The Arctic University of Norway, campus Narvik,}
\centerline{P.O. Box 385,}
\centerline{Narvik 8505, Norway}

\bigskip\bigskip
\centerline{\supind{2} Higher School of Modern Mathematics MIPT,}
\centerline{Klimentovski per., 1,}
\centerline{Moscow, 115184, Russia}

\bigskip
\centerline{\supind{3} St. Petersburg State University,}
\centerline{Universitetskaya nab.,  7/9,}
\centerline{St. Petersburg, 199034, Russia}

\bigskip

\centerline{e-mail: elena.jijina@gmail.com}
\centerline{e-mail: apiatnitski@gmail.com}
\centerline{e-mail: v.slouzh@spbu.ru}
\centerline{e-mail: t.suslina@spbu.ru}

\bigskip
\bigskip
\centerline{Abstract}

\medskip
The paper studies homogenization problem for a bounded in $L_2(\mathbb R^d)$ convolution type operator
${\mathbb A}_\eps$, $\eps >0$, of the form
$$
({\mathbb A}_\eps u) (\x) = \eps^{-d-2} \int_{\R^d} a((\x-\y)/\eps) \mu(\x/\eps, \y/\eps) \left( u(\x) - u(\y) \right)\,d\y.
$$
It is assumed that $a(\x)$ is a non-negative function from  $L_1(\R^d)$,  and $\mu(\x,\y)$ is a periodic in $\x$ and $\y$
function such that  $0< \mu_- \leqslant \mu(\x,\y) \leqslant \mu_+< \infty$.
No symmetry assumption on $a(\cdot)$ and $\mu(\cdot)$ is imposed, so the operator ${\mathbb A}_\eps$
need not be self-adjoint.
Under the assumption that the moments $M_k = \int_{\R^d} |\x|^k a(\x)\,d\x$, $k=1,2,3$, are finite
we obtain, for small $\eps>0$, sharp in order approximation of the resolvent  $({\mathbb A}_\eps + I)^{-1}$
in the operator norm in $L_2(\mathbb R^d)$, the discrepancy being of order $O(\eps)$.
The approximation is given  by an operator of the form
$({\mathbb A}^0 + \eps^{-1} \langle \boldsymbol{\alpha},\nabla  \rangle + I)^{-1}$ multiplied
on the right by a periodic function  $q_0(\x/\eps)$; here
${\mathbb A}^0 = - \operatorname{div}g^0 \nabla$
is the effective operator, and $\boldsymbol{\alpha}$ is a constant vector.

\bigskip\bigskip
\noindent\textbf{Keywords}:
Convolution type operators, periodic homogenization, operator estimates of discrepancy,
effective operator.

\bigskip\bigskip

\noindent
The work of  A.~Piatnitski  and E.~Zhizhina was partially supported by the project ''Pure Mathematics in Norway'' \ and the UiT Aurora project MASCOT.

\noindent
The research of  V.~Sloushch  and T.~Suslina was supported by the RSF foundation, project \hbox{22-11-00092-P}.

\newpage

\section*{Introduction}

\noindent{\sl Studied problem and description of results}

In this work we consider periodic homogenization problem for non-symmetric zero order convolution type operators in
$L_2(\mathbb R^d)$ of the form
\begin{equation}
\label{Sus1_intro}
 \mathbb{A}_\eps u( \x)=\frac1{\eps^{d+2}}\int\limits_{\mathbb R^d} a\Big(\frac{\x-\y}{\eps}\Big)\mu\Big(\frac \x\eps,\frac \y\eps\Big)\big( u(\x)-u(\y)\big)\,d\y,\ \ \x\in\mathbb{R}^{d},\ \ u\in L_{2}(\mathbb{R}^{d}),
 \end{equation}
with a small parameter $\eps>0$.
We assume that $a(\cdot)\geqslant 0$, $a\in L_1(\mathbb R^d)$, and $\mu(\x,\y)$ is a $\mathbb Z^d$ periodic in each variable function which is bounded and uniformly positive.   Furthermore, we assume that $a(\cdot)$ has finite moments up to order 3:
$M_{k}(a) =\int_{\mathbb{R}^{d}}|\x|^{k}a(\x)\,d\x<+\infty$, $k=1,\,2,\,3$.


The interest to nonlocal operators of convolution type and to homogenization problems for such operators is motivated
in particular by various applications in the mathematical biology,  population dynamics, mechanics of porous media and chemistry of polymers, see the introduction to \cite{PSlSuZh} for the further details.
On the other hand, the theory of these operators raises many challenging mathematical questions that attract the attention of
mathematicians.
The qualitative and asymptotic properties of convolution type operators have been actively studied
in the recent time.

 Our goal is to  approximate, for small $\eps>0$,  the
resolvent of the original operator $\mathbb A_\eps$ by means of the resolvent of an operator constructed in terms of the effective characteristics of the problem.
It should be emphasized that in the case of non-self-adjoint operators, the effective characteristics include not only the effective matrix but also the effective drift and the kernel of the adjoint periodic operator.

We show that the resolvent $(\mathbb A_\eps+I)^{-1}$ is asymptotically close,
{\it in the operator norm in} $L_2(\mathbb R^d)$, to the
operator $(\mathbb{A}^{0}+ \eps^{-1} \langle \boldsymbol{\alpha}, \nabla \rangle +I)^{-1}[q_0\big(\frac \x\eps\big) ]$,
where $\mathbb{A}^{0}=\mathrm{div}g^0\nabla$ is the {\it effective elliptic diffusion operator} with constant coefficients, $\boldsymbol{\alpha}\in\mathbb R^d$ is the so-called {\it effective
drift}, the $\mathbb Z^d$-periodic function $q_0(\x)$ defines the {\it kernel of the adjoint periodic operator} $\mathbb A^*$ corresponding to $\eps=1$:
$$
\mathbb A^*q_0(\x):=\int_{\mathbb R^d}\big( a(\x-\y)\mu(\x,\y)q_0(\x)-
a(\y-\x)\mu(\y,\x) q_0(\y)\big)d\y=0,\quad q_0\in L_2([0,1)^d),
$$
 and
$[q_0\big(\frac\x\eps\big) ]$ stands for the operator of multiplication by the function $q_0\big(\frac \x\eps\big) $.
The kernel of the periodic operator $\mathbb A^*$ is one-dimensional, and, under the normalization condition
$\int_{[0,1)^d}q_0(\y)d\y=1$, the function $q_0$ is uniquely defined and positive.

Then we prove, and this is the main result of the work, that, for the difference between the resolvent of $\mathbb A_\eps$
and the constructed approximation,  the following sharp in order estimate in the operator norm in $L_2(\mathbb R^d)$
is fulfilled:
\begin{equation}\label{0.3a_intr}
\bigl\|(\mathbb{A}_{\varepsilon}+I)^{-1}-(\mathbb{A}^{0}+ \eps^{-1} \langle \boldsymbol{\alpha}, \nabla \rangle +I)^{-1} [q_0^\eps ] \bigr\|_{L_{2}(\mathbb{R}^{d})\to
L_{2}(\mathbb{R}^{d})}\leqslant C(a,\mu)\varepsilon,\ \ \varepsilon>0,
\end{equation}
with $q_0^\eps(\x)=q_0\big(\frac \x\eps\big)$.

Notice that, in contrast with the symmetric operators, in the non-symmetric case an approximation of the resolvent of $\mathbb A_\eps$
{\sl is not reduced to just taking the resolvent of the homogenized operator}, the approximation also includes a large first order term and multiplication by a rapidly oscillating periodic function. When homogenizing the corresponding Cauchy problem
$$
\partial_t u= - \mathbb A_\eps u \ \ \hbox{in }\mathbb R^d\times(0,T],\quad u(\x,0)=u_0\in L_2(\mathbb R^d),
$$
the large first order term $\eps^{-1}\langle \boldsymbol{\alpha}, \nabla \rangle$ in \eqref{0.3a_intr} is responsible for homogenization in a moving frame $(\x,t)\to \big(\x-\frac { \boldsymbol{\alpha}}\eps t,t\big)$.
This means in particular that neither the resolvent of $\mathbb A_\eps$ nor the corresponding semigroup converges in the usual sense
to the resolvent (semigroup) of the limit operator.

\medskip
\noindent{\sl Some existing homogenization results for non-symmetric differential operators}

Various homogenization problems for differential operators have been widely studied during last fifty years. We mention here
only some key monographs in the field, see  \cite{BaPa, BeLP, JKO}.

Parabolic problems for non-symmetric convection-diffusion type differential operators in media with a
periodic microstructure have been studied in \cite{Garn}, \cite{DoPi07}, \cite{AlOr} and some other works.
It was proved 
that, in a properly chosen {\it moving frame},  the corresponding semigroups converge strongly in $L_2(\mathbb R^d)$ to the limit semigroup corresponding to the effective diffusion operator with constant coefficients.
The papers  \cite{DoPi07} and \cite{AlOr} dealt with the diffusive scaling, while  \cite{Garn} investigated the scaling
$(\x,t)\to (\eps \x,\eps^p t)$ with $p\in (0,2)$.

%
%
%

 \medskip\noindent
 {\sl Estimates for the rate of convergence}

Obtaining estimates for the rate of convergence in various averaging procedures plays
a very important role in engineering and applied sciences on the one hand, and leads to challenging mathematical problems
on the other hand.
The first quantitative homogenization  results for periodic differential operators were obtained in the monograph \cite{BeLP}, see also \cite{BaPa}, \cite{OlShYo}  and the bibliography in these books. 

Over the next 25 years, a large number of works were devoted to this topic,  where an essential progress
was achieved in estimating the rate of convergence
for various homogenization problems, including those for elasticity system,
Stokes equation, nonlinear differential equations, etc.

However, the mentioned works provide estimates for the rate of convergence  in the strong topology,
these estimates are valid  only for special classes of the right-hand sides.

In the works \cite{BSu1}, \cite{BSu3}, \cite{BSu4}  M. Birman and T. Suslina introduced and developed an operator-theoretic approach to homogenization of periodic differential operators. With the help of this approach
the so-called operator estimates for the rate of convergence have been obtained for a wide class of periodic
homogenization problems.  Let us illustrate the main ideas of the approach using as an example the homogenization of an elliptic operator
$A_\varepsilon = - \operatorname{div} g(\mathbf{x} / \varepsilon) \nabla$ in $L_2(\mathbb{R}^d)$. Here the matrix function $g(\mathbf{x})$ is assumed to be bounded, positive definite, and
$\mathbb{Z}^d$-periodic. As was shown in \cite{BSu1},  the resolvent $(A_\varepsilon +I)^{-1}$ converges, as  $\varepsilon \to 0$, in the operator norm in  $L_2(\mathbb{R}^d)$ to the resolvent of the operator  $A^0$, where
$A^0=\mathrm{div}g_{\rm hom}\nabla$ is the effective elliptic operator with constant coefficients.
  Moreover, for the discrepancy  the following estimate holds:
\begin{equation}
\label{BSu1}
\| (A_\varepsilon +I)^{-1} - (A^0 +I)^{-1}\|_{L_2(\mathbb{R}^d) \to L_2(\mathbb{R}^d)} \leqslant C \varepsilon.
\end{equation}
In the homogenization theory the estimates of this type are called {\it operator estimates} for the discrepancy.

More precise approximation of order $\eps^2$ that takes into account the corrector was constructed in \cite{BSu3}.

The method of proving \eqref{BSu1} relies on the scaling transformation, the Floquet-Bloch theory and the analytic perturbation theory.

The unitary scaling transformation reduces estimate \eqref{BSu1} to the inequality
\begin{equation}
\label{BSu2}
\| (A + \eps^2 I)^{-1} - (A^0 + \eps^2 I)^{-1}\|_{L_2(\R^d) \to L_2(\R^d)} \leqslant C \eps^{-1};
\end{equation}
here $A=A_1=- \operatorname{div} g(\mathbf{\x}) \nabla$. The operator $A$ can be expanded, by means of the unitary
Gelfand transform, into a direct integral over the operator family $A(\boldsymbol{\xi})$ in $L_2(\Omega)$, where
$\Omega$ is the unit cell of the lattice $\mathbb Z^d$, the parameter $\boldsymbol{\xi}$ called quasi-momentum
takes on values in the cell of the dual lattice $\tilde\Omega=[-\pi,\pi)^d$, and
$A(\boldsymbol{\bxi})=(-i\nabla+\bxi)^*g(\x)(-i\nabla+\bxi)$ with periodic boundary conditions.
In order to justify \eqref{BSu2} it is sufficient to show that
\begin{equation*}
\| (A(\bxi) + \eps^2 I)^{-1} - (A^0(\bxi) + \eps^2 I)^{-1}\|_{L_2(\Omega) \to L_2(\Omega)} \leqslant C \eps^{-1},
\quad\hbox{for all } \bxi \in \widetilde{\Omega}.
\end{equation*}
The analysis of the family of operators  $A(\bxi)$ is a crucial part of the approach.
Since $\{A(\bxi)\}$ is an analytic  family of operators with compact resolvent, the methods of
analytic perturbation theory can be used.  It was shown that the resolvent $(A(\bxi) + \eps^2 I )^{-1}$
can be approximated in terms of the spectral characteristics of the operator at its spectral edge.  In particular,
the effective matrix coincides with the Hessian of the first eigenvalue $\lambda_1(\bxi)$ of $A(\bxi)$ at $\bxi=0$.
Therefore,  homogenization can be treated as a {\sl spectral threshold effect} at the spectral edge
of an elliptic operator.

 Another method of deriving operator estimates in periodic homogenization problems
 was introduced by V. Zhikov and S. Pastukhova in  \cite{Zh, ZhPas1}, see also the survey \cite{ZhPas3} and the bibliography there. This method is called ''shift method'', it is also applicable to quantitative  homogenization problems in locally periodic environments and in bounded domains.

 Recent years the problem of obtaining operator estimates for the discrepancy in various homogenization problems
 attracted the attention of many mathematicians. A number of deep results has been obtained
 in this area. We refer to   \cite[Section 0.2]{PSlSuZh} and \cite[Introduction]{Su_UMN2023},
where the detailed description of the state of arts in this field can be found.

 \medskip\noindent
 {\sl Homogenization of nonlocal convolution type operators}

Although averaging problems for differential operators have been studied for a quite long period of time, the first
homogenization results for zero order convolution type operators was obtained 
quite recently. It was proved in     \cite{PZh} that, under natural coerciveness and moment condition, the family of operators $\mathbb A_\eps$ in \eqref{Sus1_intro} with symmetric coefficients admits homogenization in a periodic medium. Moreover, the effective
operator is a second order elliptic differential operator with constant coefficients. Similar results in perforated domains
were obtained in \cite{BraPia21} by the variational methods. It is interesting to observe that  although
 the operator $\mathbb A_\eps$ is nonlocal and bounded for each $\eps>0$, the homogenized operator is local and
  unbounded.

  The case of non-autonomous convolution type parabolic equations whose coefficients oscillate periodically
 both in spatial and temporal variables, was addressed in \cite{PZh23}.

  Homogenization problems for non-symmetric zero order
 convolution type operators were considered in \cite{PiaZhi19}. As in the case of differential operators, the homogenization result for the corresponding parabolic equation holds in moving coordinates.

 A  number of periodic homogenization results for L\'evy type operators corresponding to stable-like
  Markov processes were obtained in the recent works \cite{KPZ19} and \cite{CCKZ21_2}.
  The paper  \cite{CCKZ21}  focuses on estimating the convergence rate for these operators in the strong topology.
  The approach used in  \cite{CCKZ21} relies on probabilistic arguments.
  Operator estimates for  this homogenization  problem are the subject of work \cite{PSSZ_le}.

\medskip
The present work deals with periodic homogenization of convolution type operators of the form \eqref{Sus1_intro}.
We obtain sharp in order operator estimates for the rate of convergence in $L_2(\mathbb R^d)$.  In the symmetric case
this problem was addressed in the authors' works \cite{PSlSuZh, PSlSuZh1, PSlSuZh2}.
We emphasize that in all the previous works on homogenization of periodic problems the operator-theoretic approach
applied only to self-adjoint operators, both in the case of differential and of non-local operators. It was not clear if this approach can be generalized to  non-self-adjoint problems.  In the present paper
we succeeded to adapt the operator-theoretic approach to the case of non-symmetric convolution type operators.
We strongly believe that the homogenization technique developed in this paper will also be applicable to other
 non-self-adjoint problems.

It should also be noted that before the appearance of the authors' work \cite{PSlSuZh}, operator estimates for non-local convolution type operators had not been studied.

\subsection{Methods}
The method of approximating the resolvent  $(\mathbb A_\eps+I)^{-1}$,  
as $\eps\to0$, relies on a modified version of the operator-theoretic approach developed in  \cite{PSlSuZh, PSlSuZh1, PSlSuZh2}, which is adapted to the case of nonlocal operators.

The first two steps of this approach, namely the unitary scaling transformation and decomposing $\A$ into a direct integral
over the operator family $\mathbb{A}(\bxi)$ by means of the Gelfand transform, remain unchanged.
Then the problem is reduced to studying the asymptotic behaviour of the resolvent   $(\mathbb{A}(\bxi)+\eps^{2}I)^{-1}$
for small $\eps>0$, this resolvent should be approximated with the precision $O(\eps^{-1})$.
The operators  $\mathbb{A}(\bxi)$ are defined in Section \ref{sec1.2} below, they are bounded operators in $L_2(\Omega)$
and depend on a parameter $\bxi \in \widetilde{\Omega}$.
However, in contrast with the differential operators, the family $\{\mathbb{A}(\bxi)\}$ is not analytic, and thus
the methods of analytic perturbation theory do not apply to this family. Instead, we use the finite smoothness of
$\mathbb{A}(\bxi)$ which is granted by the assumption that the coefficient $a(\x)$ has finite moments up to order 3.



We show that the spectrum of $\mathbb{A}(\bxi)$ is located in the half-plane  $\{\lambda\in\mathbb C\,:\,\operatorname{Re} \lambda \ge 0\}$.
Moreover, for $|\bxi| > \delta > 0$ it belongs to the half-plane $\operatorname{Re} \lambda \ge C(\delta) >0$.
Therefore, it is sufficient to determine the asymptotics of the resolvent    $(\mathbb{A}(\bxi)+\eps^{2}I)^{-1}$
for $\bxi$ from the set  $|\bxi| \le \delta_0$, where $\delta_0>0$ is a small enough constant.
Further, it is clear that only the spectral characteristics of the operator $\mathbb{A}(\bxi)$ near the point $\lambda_0=0$
are responsible for the non-negligible contribution to the studied asymptotics. Indeed, letting $F(\bxi)$  be
the Riesz projector of the operator $\mathbb{A}(\bxi)$ that corresponds to some neighbourhood of zero, we show
that  the operator $(\mathbb{A}(\bxi)+\eps^{2}I)^{-1}(I - F(\bxi))$ is uniformly in $\eps$ bounded. Consequently,
it is sufficient to approximate the operator  $(\mathbb{A}(\bxi)+\eps^{2}I)^{-1}F(\bxi)$.  Here, the crucial role
is played by the so-called threshold approximations -- the expansions of the operators $F(\bxi)$ and $\mathbb{A}(\bxi) F(\bxi)$ for small $\bxi$.


We consider $\mathbb{A}(\bxi)$ as a perturbation of the operator $\mathbb{A}(\mathbf{0})$. It turns out that
$\lambda_0 =0$ is a point of the discrete spectrum of $\mathbb{A}(\mathbf{0})$ and that the kernel
$\operatorname{Ker} \mathbb{A}(\mathbf{0})$ is one-dimensional and consists of constants.
The kernel of the adjoint operator $\operatorname{Ker} \mathbb{A}^*(\mathbf{0})$ is also one-dimensional
and consists of functions of the form $c\, q_0(\x)$, $c \in \C$, where $ q_0(\x)$ is such that
$0<q_-\leqslant q_0(\x) \leqslant q_+$. Then for  $|\bxi| \le \delta_0$ the spectrum of $\mathbb{A}(\bxi)$
in a small enough neighbourhood of $\lambda=0$ consists of a one simple eigenvalue $\lambda_1(\bxi)$.

We choose a smooth contour $\Gamma \subset \C$ that comprises the eigenvalue   $\lambda_1(\bxi)$
and is separated from the remaining part of the spectrum, and then integrate the resolvent  $(\A(\bxi) - \zeta I)^{-1}$
over this contour in order to calculate the asymptotics of the operators  $F(\bxi)$ and $\mathbb{A}(\bxi)F(\bxi)$, as $\bxi\to0$.
In the symmetric case this technique was used in \cite[Section 4.2]{PSlSuZh}.


\subsection{Structure of the paper.}
The paper consists of the Introdiction and six Secrions.

In \S1 we introduce the operator  ${\mathbb A}$, decompose this operator into a direct integral over the
 family of operators   ${\mathbb A}(\bxi)$ and study the spectral characteristics of the operator
 $\mathbb{A}(\mathbf{0})$  for small values of the spectral parameter. Also, we provide a lower bound
 for the quadratic form of the operator  $\operatorname{Re}[q_0]{\mathbb A}(\bxi)$.
 In \S2 the threshold approximations of the operators $F(\bxi)$ and $\mathbb{A}(\bxi) F(\bxi)$  are obtained
for small $|\bxi|$.

Section 3 focuses on the asymptotic behaviour of the resolvent $( {\mathbb A}(\bxi) + \eps^2 I)^{-1}$ for small $\eps$,
this asymptotics is then used for approximating the resolvent $( {\mathbb A} + \eps^2 I)^{-1}$.
Section 4 focuses on approximating  the resolvent $({\mathbb A}_\eps + I)^{-1}$ in the operator norm in the space $L_2(\R^d)$.
The desired asymptotics is deduced
form the results of Section 3 by means of a scaling transformation. This is the main result of the paper.
 Sections 5 and 6 contain auxiliary statements on the stability of an isolated eigenvalue of ${\mathbb A}(\bxi)$ (\S5)
 and on the properties of the auxiliary operator  $\mathbf G$ (\S 6).

\subsection{Notation}
The norm in a normed space $X$ is denoted by $\|\cdot\|_{X}$, or just $\|\cdot\|$, if it does not lead
to an ambiguity; for two normed spaces  $X$ and $Y$ the standard norm of a bounded linear operator
$T:X\to Y$ is denoted by $\|T\|_{X\to Y}$  or just  $\|T\|$. The notation  $\mathcal{L}\{F\}$ stands for
the span of a set of vectors $F\subset X$.
The space of bounded linear operators in a normed space $X$ is denoted by $\mathcal{B}(X)$.

Given complex separable Hilbert spaces $\H$ and $\H_*$ and a linear operator \hbox{$A: \H \to \H_*$},
the notation  $\operatorname{Dom} A$  stands for the domain of this operator and $\operatorname{Ker} A$ --
for its kernel.
For a closed linear operator $\A$ in a Hilbert space
$\mathfrak{H}$ the spectrum of  $\A$ is denoted by $\sigma(\A)$.

If $\mathcal O$ is a domain in  $\R^d$, then the notation $L_{p}({\mathcal O})$,
\hbox{$1 \le p \le \infty$}, is used for the standard $L_p$ spaces in $\mathcal O$.
The standard scalar product in the space $L_{2}({\mathcal O})$ is denoted
by $(\cdot,\cdot)_{L_{2}({\mathcal O})}$ or without the index.
For $f\in L_\infty({\mathcal O})$ the symbol  $[f]$ stands for the operator of multiplication
by the function $f(\x)$ in the space  $L_{2}({\mathcal O})$.
The notation $H^s({\mathcal O})$ is used for the Sobolev  spaces of order $s>0$ in the domain $\mathcal O$,
$\mathcal{S}(\R^{d})$ stands for the Schwartz class in $\R^{d}$.

The scalar product in $\R^d$ and $\mathbb{C}^{d}$ is denoted by $\langle\cdot,\cdot\rangle$.
In what follows we also use the notation $\mathbf{x} = (x_1,\dots, x_d)^{t} \in \R^d$, $i D_j = \partial_j = \partial / \partial x_j$, $j=1,\dots,d$; $\mathbf{D} = - i \nabla = (D_1,\dots,D_d)^t$.
The characteristic function of a set $E\subset\R^d$ is denoted by  $\mathbf{1}_{E}$,
the circle  $\{z \in \C: |z-z_0|< r\}$ is denoted by $B_r(z_0)$.\\[3mm]

{
\section{Convolution type  operators: \\ representation as a direct integral. Estimates.}

\subsection{Operator $\A(a,\mu)$}}
Given functions $a\in L_{1}(\R^d)$ and  $\mu\in L_{\infty}(\R^d\times \R^d)$,
we define in $L_{2}(\R^d)$  a nonlocal {\sl convolution type operator } $\A = \A(a,\mu)$  by
\begin{equation*}
\A(a,\mu) u(\x):=\intop_{\R^d}a(\x-\y)\mu(\x,\y)(u(\x)-u(\y)) \,d\y,\ \ \x\in\R^d.
\end{equation*}
The operator $\A$ can be represented as follows: $\A=p(\cdot;a,\mu)-\B(a,\mu)$, where
\begin{equation*}
\begin{gathered}
p(\x;a,\mu):=\intop_{\R^d}a(\x-\y)\mu(\x,\y) \,d\y,\ \ \x\in\R^d,\\
\B(a,\mu) u(\x):=\intop_{\R^d}a(\x-\y)\mu(\x,\y)u(\y)\,d\y,\ \ \x\in\R^d.
\end{gathered}
\end{equation*}
According to the Schur lemma, see for instance  \cite[Lemma 4.1]{PSlSuZh},
the operator \hbox{$\B(a,\mu):L_{2}(\R^d)\to L_{2}(\R^d)$} is bounded and its norm in $L_2(\mathbb R^d)$
satisfies the estimate  $\|\B(a,\mu)\|\le\|\mu\|_{L_\infty}\|a\|_{L_1}$.
By the definition of $p(\x;a,\mu)$ we have
$\|p(\cdot;a,\mu)\|_{L_\infty}\le\|\mu\|_{L_\infty}\|a\|_{L_1}$.
Therefore, the operator $\A(a,\mu):L_{2}(\R^d)\to L_{2}(\R^d)$ is bounded in $L_2(\mathbb R^d)$.

In the case $\mu=\mu_{0}\equiv 1$ denote
   $\A_{0}(a):=\A(a,\mu_{0})$;
   $p_{0}(\x;a):=p(\x;a,\mu_{0})$; $\B_{0}(a):=\B(a,\mu_{0})$. Clearly, $\B_{0}(a)$ is the operator of convolution
   with the function $a$, and the potential $p_{0}(\x;a) = \int_{\R^d} a(\y)\,d\y$ is a constant.

In what follows we impose the following conditions on the functions $a$ and  $\mu$:
\begin{gather}
\label{h1.1}
a\in L_{1}(\R^d),\ \ \operatorname{mes}\{\x\in\R^d:a(\x)\not=0\}>0,\ \ a(\x)\ge 0,\ \ \x\in\R^{d};
\\
\label{h1.2}
0<\mu_{-}\le\mu(\x,\y)\le\mu_{+}<+\infty,\ \ \x, \y\in\R^d;
\\
\label{h1.3}
\mu(\x+\m,\y+\n)=\mu(\x,\y),\ \ \x, \y\in\R^d,\ \ \m,\n\in\Z^d.
\end{gather}
The notation $M_{k}(a)$ stands for the moment of order $k$ of the function $a$:
\begin{equation*}
M_{k}(a):=\intop_{\R^d}| \x|^{k}a(\x)\, d\x,\ \ k \in \N.
\end{equation*}
Since $0 < \int_{\R^d} a(\x)\, d\x < \infty$, the finiteness of the moment  $M_k(a)$
implies the finiteness of the moments $M_1(a),\dots, M_{k-1}(a)$.
The main result of this work, Theorem  \ref{teor3.1}, holds under the assumption $M_3(a) < \infty$.

By conditions \eqref{h1.1}--\eqref{h1.3}, the potential $p(\x)=p(\x;a,\mu)$ is real, $\Z^{d}$-periodic and satisfies the
estimates
\begin{equation}
\label{h1.4} \mu_{-}\|a\|_{L_1(\R^d)} \le
p(\x)\le\mu_{+}\|a\|_{L_1(\R^d)},\ \ \x\in\R^d.
\end{equation}

\subsection{Representation of the operator $\A(a,\mu)$ as a direct integral}\label{sec1.2}
Since the function $\mu(\x,\y)$ is $\mathbb Z^d\times\mathbb Z^d$-periodic, both the operator of multiplication
by the potential  $p(\x;a,\mu)$ and the operator $\B(a,\mu)$, and thus the operator $\A(a,\mu)$, commute
with the operators  $S_{\n}$ of integer shifts. The latter operators are defined by
%
\begin{equation*}
S_{\n}u(\x) =u(\x+\n),\ \ \x\in\R^d,\ \ \n\in\Z^d.
\end{equation*}
This means that both $\A(a,\mu)$ and $\B(a,\mu)$ are periodic operators with the periodicity lattice $\Z^d$. Denote
by $\Omega:=[0,1)^{d}$ the cell of the lattice $\Z^d$ and by  $\widetilde\Omega:=[-\pi,\pi)^{d}$ the cell of the dial
lattice $(2\pi\mathbb{Z})^{d}$.

We recall the definition of the Gelfand transform, see, for instance,  \cite{Sk} and \cite[Ch.~2]{BSu1}.
First the Gelfand transform  $\mathcal{G}$ is defined for the functions from the Schwartz class ${\mathcal S}(\R^d)$
by the formula
\begin{equation*}
\mathcal{G}u(\bxi,\x):=(2\pi)^{-d/2}\sum_{\n\in\Z^d} u(\x+\n) e^{-i \langle \bxi, \x+\n\rangle},\
\ \bxi\in\widetilde\Omega,\ \ \x\in\Omega,\ \ u\in {\mathcal S}(\R^d).
\end{equation*}
%
Then  $\mathcal{G}$ is extended by continuity to a unitary mapping
 $$
 L_{2}(\R^d) \to \int_{\widetilde\Omega}\oplus L_{2}(\Omega)\, d\bxi=L_{2}(\widetilde\Omega\times\Omega).
 $$

Like any periodic operator, $\A(a,\mu)$ and $\B(a,\mu)$ admit a decomposition into a direct integral
by means of the Gelfand transform, that is these operators admit a partial diagonalization:
\begin{equation}
\label{h1.5} \A(a,\mu) = {\mathcal G}^* \Bigl(
\int_{\widetilde\Omega} \oplus  \A(\bxi;a,\mu) \,d\bxi\Bigr)
{\mathcal G},\quad \B(a,\mu) = {\mathcal G}^* \Bigl(
\int_{\widetilde\Omega} \oplus  \B(\bxi;a,\mu) \,d\bxi\Bigr)
{\mathcal G}.
\end{equation}
Here $\A(\bxi) = \A(\bxi; a,\mu)$ and $\B(\bxi) = \B(\bxi; a,\mu)$ are bounded linear operators in
$L_{2}(\Omega)$ defined by
\begin{align}
\label{A_xi}
\A(\bxi; a,\mu) u (\x) &= p(\x;a,\mu) u(\x) - \B(\bxi; a,\mu)u(\x), \ \ u\in L_{2}(\Omega),
\\
\label{B_xi}
\B(\bxi; a,\mu)u(\x) &= \intop_{\Omega}\widetilde a(\bxi,\x-\y)\mu(\x,\y)u(\y)\,d\y,\ \ u\in L_{2}(\Omega),
\end{align}
with
\begin{equation}
\label{a_tilde}
\widetilde{a}(\bxi,\z) :=\sum_{\n\in\Z^d}a(\z+\n)e^{-i \langle \bxi, \z+\n \rangle },\ \
\bxi\in\widetilde\Omega,\ \ \z\in\R^d,
\end{equation}
and
\begin{equation}\label{e1.9}
p(\x;a,\mu)=\intop_{\Omega}\widetilde a(\mathbf{0},\x-\y)\mu(\x,\y)\,d\y.
\end{equation}
To clarify the first relation in \eqref{h1.5}, we consider an arbitrary $u
\in L_2(\R^d)$ and denote $v= \A u$.  Then $\mathcal{G} v (\bxi,
\cdot) = \A(\bxi) \mathcal{G} u (\bxi, \cdot)$, $\bxi \in
\wt{\Omega}$. The second relation is understood in the same way.

The operator $\B(\bxi; a,\mu)$ is compact in $L_2(\Omega)$, see \cite[Corollary 4.2]{PSlSuZh};
by the Schur lemma the norm of this operator admits the estimate
\begin{equation*}
\|\B(\bxi; a,\mu)\|  \le \mu_{+}\|a\|_{L_1(\R^d)},\ \ \bxi\in\widetilde\Omega.
\end{equation*}
Let $\sigma([p])$ be the spectrum of the operator of multiplication by the potential $p(\x;a,\mu)$.
Then $\sigma([p])\subset[\mu_{-}\|a\|_{L_{1}},\mu_{+}\|a\|_{L_{1}}]$, and
$$
\A(\bxi; a,\mu)-\lambda I=([p]-\lambda)(I-([p]-\lambda)^{-1}\B(\bxi; a,\mu)),\ \ \lambda\in\mathbb{C}\setminus\sigma([p]),\ \ \bxi\in\widetilde\Omega.
$$
Clearly, the operator $I\!-\!([p]\!-\!\lambda)^{-1}\B(\bxi; a,\mu)$ is invertible, if
$\mathrm{dist}(\sigma([p]),\lambda)>\|\B(\bxi; a,\mu)\|$.
Therefore, by the analytic Fredholm theorem, see, for instance, \cite[Ch. XI, Corollary 8.4]{GS1}, the operator
 $I-([p]-\lambda)^{-1}\B(\bxi; a,\mu)$ is invertible for all  $\lambda\in\mathbb{C}\setminus\sigma([p])$
 except for an at most countable set of points $\{\lambda_{k}(\bxi)\}$ which can only accumulate to points of
 $\sigma([p])$;  moreover, in a small enough neighbourhood of any such  point $\lambda_{k}(\bxi)$
 the following expansion is valid:
$$
(I-([p]-\lambda)^{-1}\B(\bxi; a,\mu))^{-1}=\sum_{n=-q_{k}}^{\infty}(\lambda-\lambda_{k}(\bxi))^{n}A_{n}^{(k)}(\bxi),
$$
where $A_{-q_{k}}^{(k)}(\bxi),\dots,A_{-1}^{(k)}(\bxi)$ are
operators of finite range.
Therefore, in the set $\mathbb C\setminus \sigma([p])$ the spectrum of the operator $\A(\bxi; a,\mu)$
coincides with the set of points $\{\lambda_{k}(\bxi)\}$, and the Riesz projector of the operator $\A(\bxi; a,\mu)$
corresponding  to each of these points 
has a finite rank.
Furthermore,  the spectrum of the operator $\A^{*}(\bxi; a,\mu)$ in the complement of $\sigma([p])$ in $\mathbb C$
consists of the points $\{\overline{\lambda_{k}(\bxi)} \}$, and the Riesz projectors of $\A^{*}(\bxi;a,\mu)$ corresponding
to these points 
also have a finite rank.

\subsection{The spectrum of the operators $\A({\mathbf 0};a,\mu)$ and $\A^{*}({\mathbf 0};a,\mu)$
 in the vicinity of zero}
By  \eqref{A_xi}, \eqref{B_xi}, \eqref{e1.9} we have $\A({\mathbf 0};a,\mu)\mathbf{1}_{\Omega}=0$.
Consequently, the point $\lambda_{0}=0$
is an element of the spectrum $\sigma(\A({\mathbf 0};a,\mu))$ and thus of the spectrum
$\sigma(\A^{*}({\mathbf 0};a,\mu))$. According to \eqref{h1.1}, \eqref{h1.4} the point  $\lambda_{0}=0$ is at a positive distance from  $\sigma([p])$; therefore,
$\lambda_{0}=0$ is an isolated point of  $\sigma(\A({\mathbf 0};a,\mu))$ and
$\sigma(\A^{*}({\mathbf 0};a,\mu))$.  Denote by $P$ and $P^{*}$ the Riesz projectors of the operators
$\A({\mathbf 0}; a,\mu)$ and $\A^{*}({\mathbf 0}; a,\mu)$, respectively, that correspond to the point $\lambda_{0}=0$.

\begin{proposition}\label{YadroA}
Under conditions  \eqref{h1.1}--\eqref{h1.3} the following statements hold true:
\begin{enumerate}
\item[1)]  The Riesz projectors $P$ and $P^{*}$ have rank $1$\textup{;}

\item[2)] The kernels of $\A({\mathbf 0}; a,\mu)$ and $\A^{*}({\mathbf 0};a,\mu)$ are given by
\begin{equation*}\label{KerA}
\operatorname{Ker}\A({\mathbf 0}; a,\mu)=\mathcal{L}\{\mathbf{1}_{\Omega}\},\ \ \operatorname{Ker}\A^{*}({\mathbf 0};a,\mu)=\mathcal{L}\{q_{0}\},\ \ q_{0}\in L_{2}(\Omega);
\end{equation*}

\item[3)]  The function $q_{0}$ can be chosen in such a way that
\begin{equation}\label{d1.0.0}
0<q_{-}\le q_{0}(\x)\le q_{+}<+\infty,\ \ \x\in\Omega;\ \ \int_{\Omega}q_{0}(\x) \, d\x=1.
\end{equation}
\end{enumerate}
\end{proposition}

\begin{proof}
We define a compact operator $\mathbf{G}=\B^*({\mathbf 0}; a,\mu)[p^{-1}]$ and notice that the kernel  $\operatorname{Ker}\A({\mathbf 0};a,\mu)$ coincides with the kernel $\operatorname{Ker}(\mathbf{G}^{*}-I)$.
The kernel $\operatorname{Ker}\A^{*}({\mathbf 0};a,\mu)$ is isomorphic to the kernel
$\operatorname{Ker}(\mathbf{G}-I)$ in the following sense:
$$
u\in \operatorname{Ker}\A^{*}({\mathbf 0}; a,\mu) \Longleftrightarrow pu\in \operatorname{Ker}(\mathbf{G}-I).
$$
Below in the Appendix we check, see Theorem \ref{theorem1.5}, that
\begin{equation}\label{f000}
\operatorname{Ker}(\mathbf{G}^{*}-I)=\mathcal{L}\{\mathbf{1}_{\Omega}\},\ \
\operatorname{Ker}(\mathbf{G}-I)=\mathcal{L}\{\psi_{0}\};
\end{equation}
here the function $\psi_{0}\in L_{2}(\Omega)$ satisfies the following relations:
\begin{equation}\label{f100}
0<\psi_{-}\le\psi_{0}(\z)\le\psi_{+}<+\infty,\ \ \z\in\Omega;\ \ \int_{\Omega}\psi_{0}(\z)\,d\z=1.
\end{equation}
Now the statements 2) and 3)  immediately follow from  \eqref{f000} and \eqref{f100},
the function  $p(\z) q_{0}(\z)$ being proportional to $\psi_{0}(\z)$.

Our next goal is to show that the rank of the Riesz projector $P$ is equal to $1$. Since  $P$ has a finite rank,
$\operatorname{Ran}P$ coincides with the root subspace  $\cup_{k=1}^{\infty}\operatorname{Ker}\A^{k}({\mathbf 0}; a,\mu)$. Consequently, it is sufficient to check that
$\operatorname{Ker}\A^{2}({\mathbf 0}; a,\mu)=\operatorname{Ker}\A({\mathbf 0}; a,\mu)$.
Assume by contradiction that there exists $v\in\operatorname{Ker}\A^{2}({\mathbf 0}; a,\mu)\setminus\operatorname{Ker}\A({\mathbf 0}; a,\mu)$ i. e.  $\A({\mathbf 0}; a,\mu)v=\alpha\mathbf{1}_{\Omega}$, $\alpha\not=0$. Then $\mathbf{1}_{\Omega}\in\operatorname{Ran}\A({\mathbf 0}; a,\mu)$ and thus $\mathbf{1}_{\Omega}\perp q_{0}$. The latter relation contradicts to \eqref{d1.0.0}.
\end{proof}

Note that the statements 2) and 3) of Proposition \ref{YadroA} were previously obtained in the work \cite[Corollary 4.1]{PiaZhi19}. The proof of these statements given in  \cite{PiaZhi19} relies on the Krein-Rutman theorem.
\begin{corollary}\label{YadroAP}
For the projectors $P$ and $P^{*}$ the relations
$P=(\cdot,q_{0})\mathbf{1}_{\Omega}$ and $P^{*}=(\cdot,\mathbf{1}_{\Omega})q_{0}$ are valid.
\end{corollary}
%

\begin{remark}
From now on we assume that the function $q_{0}(\z)$ is extended periodically from $\Omega$ to the whole $\R^{d}$.
\end{remark}

\begin{remark}
Using \eqref{A_xi}--\eqref{a_tilde} with $\boldsymbol{\xi} = 0$ it is easy to show that
the relation  $\A^{*}({\mathbf 0}; a,\mu)q_{0}=0$ is equivalent to each of the following two
equalities:
\begin{equation}\label{10.02.1}
p(\x)q_{0}(\x)=\int_{\Omega}\wt a({\mathbf 0},\y-\x)\mu(\y,\x)q_{0}(\y)\,d\y,\ \ \x\in\Omega,
\end{equation}
\begin{equation}\label{07.02.2}
p(\x)q_{0}(\x)=\int_{\R^{d}}a(\y-\x)\mu(\y,\x)q_{0}(\y)\,d\y,\ \ \x\in\R^{d}.
\end{equation}
\end{remark}

\begin{remark}\label{remark1.5}
If $\mu\! =\! \mu_{0}\!\equiv \! 1$,  then
$p(\x; a,\mu_{0})\!\equiv\! \int_{\R^{d}}a(\z)\,d\z$,  and $q_{0}(\x) \! \equiv \! 1$.
\end{remark}

\subsection{Estimates of the quadratic form of operator $\operatorname{Re} [q_{0}]\A(a,\mu)$}

\begin{lemma}
\label{lemma1.1}
Let conditions \eqref{h1.1}--\eqref{h1.3} be fulfilled. Then
\begin{equation}
\label{e1.11}
\operatorname{Re}([q_{0}]\A(a,\mu)u,u) = \frac{1}{2}\intop_{\R^{d}}d\x\,q_{0}(\x)\intop_{\R^d}\,d\y\,
a(\x-\y)\mu(\x,\y)|u(\x)-u(\y)|^{2},\ \ u\in L_{2}(\R^{d}).
\end{equation}
\end{lemma}

\begin{proof}
The quadratic form of the operator $[q_{0}]\A(a,\mu)$ reads
\begin{multline}\label{e1.14}
([q_{0}]\A(a,\mu)u,u)=\intop_{\R^{d}}d\x\,q_{0}(\x)\intop_{\R^d} \,d\y\,
a(\x-\y)\mu(\x,\y)\overline{u(\x)} \bigl( u(\x)
-u(\y) \bigr)
\\
=\intop_{\R^{d}}d\x\,q_{0}(\x)\int_{\R^d}d\y\, a(\x-\y)\mu(\x,\y) \bigl|u(\x)
-u(\y) \bigr|^{2}+J[u],
\ \ u\in L_{2}(\R^{d});
\end{multline}
here the functional $J[u]$ is defined by
\begin{equation}
\label{e1.15}
J [u]:=\intop_{\R^{d}}d\x\,q_{0}(\x)\intop_{\R^d}d\y\, a(\x-\y)\mu(\x,\y)\overline{u(\y)}
\bigl(u(\x)-u(\y) \bigr),\ \ u\in L_{2}(\R^{d}).
\end{equation}
Representing the right-hand side of  (\ref{e1.15}) as a sum of two terms and using \eqref{07.02.2} we obtain
\begin{multline}\label{e1.16}
J[u]=-\intop_{\R^{d}}d\x\, q_{0}(\x) \intop_{\R^d}d\y\, a(\x-\y)\mu(\x,\y)|u(\y)|^{2}+\intop_{\R^{d}}d\x\,q_{0}(\x)\intop_{\R^{d}}d\y\, a(\x-\y)\mu(\x,\y)u(\x)\overline{u(\y)}
\\
=-\intop_{\R^{d}}d\y\,|u(\y)|^{2}\intop_{\R^{d}}d\x\, q_{0}(\x)a(\x-\y)\mu(\x,\y)+\overline{\intop_{\R^{d}}d\x\,q_{0}(\x)\intop_{\R^{d}}d\y\, a(\x-\y)\mu(\x,\y)\overline{u(\x)}u(\y)}
\\
=-\overline{([q_{0}]\A(a,\mu)u,u)},\ \ u\in L_{2}(\R^{d}).
\end{multline}
The desired relation  (\ref{e1.11})  follows from (\ref{e1.14}) and (\ref{e1.16}).
\end{proof}

\begin{corollary}\label{corollary1.7}
The operator $[q_{0}^{1/2}]\A(a,\mu) [q_{0}^{-1/2}]$ is accretive, the spectrum $\sigma(\A(a,\mu))$ is
located in the right half-plane. Moreover, the following estimate holds:
\begin{equation}\label{10.02.2}
\|(\A(a,\mu)-zI)^{-1}\|\le q_{+}^{1/2}q_{-}^{-1/2}|\operatorname{Re}z|^{-1},\ \ \operatorname{Re}z<0.
\end{equation}
\end{corollary}

\begin{proof}
According to identity  (\ref{e1.11}), the operator $[q_{0}^{1/2}]\A(a,\mu) [q_{0}^{-1/2}]$ is accretive. Therefore, for all $z\in\mathbb C$
such that $\operatorname{Re}z<0$ the operator $[q_{0}^{1/2}]\A(a,\mu)[q_{0}^{-1/2}] -zI$ is invertible, and the resolvent
  $ \bigl([q_{0}^{1/2}]\A(a,\mu)[q_{0}^{-1/2}]-zI \bigr)^{-1}$ satisfies the estimate
\begin{equation}
\label{1.19a}
\bigl\| \bigl([q_{0}^{1/2}]\A(a,\mu)[q_{0}^{-1/2}]-zI \bigr)^{-1} \bigr\| \le |\operatorname{Re}z|^{-1}.
\end{equation}
Consequently, the spectrum $\sigma(\A(a,\mu))$ is located in the right half-plane.
Combining an evident relation
$$
(\A(a,\mu)-zI)^{-1} =  [q_{0}^{-1/2}] \bigl([q_{0}^{1/2}]\A(a,\mu)[q_{0}^{-1/2}]-zI \bigr)^{-1} [q_{0}^{1/2}]
$$
with \eqref{d1.0.0} and \eqref{1.19a},  we derive estimate \eqref{10.02.2}.
\end{proof}

\subsection{Representation for the quadratic form of the operator  $\operatorname{Re}[q_{0}]\A(\bxi; a,\mu)$}
\begin{lemma}
\label{lemmau1.1}
Assume that  conditions \eqref{h1.1}--\eqref{h1.3} are fulfilled. Then the following relation holds:
\begin{multline}
\label{ue1.11}
\mathrm{Re}([q_{0}]\A(\bxi; a,\mu)u,u)
\\
= \frac{1}{2}\intop_{\Omega}q_{0}(\x)\,d\x\intop_{\R^d}\,d\y\,
a(\x-\y)\mu(\x,\y) \bigl|e^{i\langle \bxi, \x\rangle}u(\x)-e^{i \langle \bxi, \y\rangle}u(\y) \bigr|^{2},\ \ u\in L_{2}(\Omega),\ \ \bxi\in\widetilde\Omega.
\end{multline}
Here we identify a function $u\in L_{2}(\Omega)$ with its periodic extension on the whole $\R^d$.
\end{lemma}

\begin{proof}
Due to \eqref{A_xi}, \eqref{B_xi} and (\ref{e1.9}) the operator $\A(\bxi; a,\mu)$ admits the representation
\begin{equation}\label{ue1.12}
\A(\bxi; a,\mu)
u(\x)=\intop_{\Omega} \left( \widetilde{a}({\mathbf 0},\x-\y)u(\x)-\widetilde{a}(\bxi,\x-\y)u(\y)\right) \mu(\x,\y)\, d\y,\
\ \x\in\Omega,\ \ u\in L_{2}(\Omega).
\end{equation}
Taking into account \eqref{a_tilde} and identifying each function $u\in L_{2}(\Omega)$ with its $\Z^d$-periodic
extension on the whole $\R^d$, we can rearrange  (\ref{ue1.12}) as follows:
\begin{multline}
\label{ue1.13}
\A(\bxi; a,\mu) u(\x)=\sum_{\n\in\Z^d}\intop_{\Omega} \bigl(a(\x-\y+\n)u(\x)-a(\x-\y+\n)
e^{-i \langle \bxi, \x-\y+\n \rangle}u(\y)\bigr)\mu(\x,\y) \,d\y
\\
= \intop_{\R^d}a(\x-\y)\mu(\x,\y) \bigl(u(\x)-e^{-i \langle \bxi, \x-\y \rangle}u(\y)\bigr) \,d\y
\\
= \intop_{\R^d}a(\x-\y)\mu(\x,\y)e^{-i \langle \bxi, \x\rangle} \bigl(e^{i \langle \bxi, \x\rangle}u(\x)
-e^{i \langle \bxi, \y \rangle}u(\y)\bigr) \,d\y,\ \ \x\in\Omega,\ \ u\in L_{2}(\Omega).
\end{multline}
This yields the following expression for the quadratic form of the operator $[q_{0}]\A(\bxi; a,\mu)$:
\begin{multline}
\label{ue1.14}
([q_{0}]\A(\bxi; a,\mu)u,u)=\intop_{\Omega}d\x\,q_{0}(\x)\intop_{\R^d} \,d\y\,
a(\x-\y)\mu(\x,\y)e^{-i \langle \bxi, \x\rangle}\overline{u(\x)} \bigl( e^{i \langle \bxi, \x \rangle}u(\x)
-e^{i \langle \bxi, \y \rangle}u(\y) \bigr)
\\
=\intop_{\Omega}d\x\,q_{0}(\x)\int_{\R^d}d\y\, a(\x-\y)\mu(\x,\y) \bigl|e^{i \langle \bxi, \x \rangle}u(\x)
-e^{i \langle \bxi, \y \rangle}u(\y) \bigr|^{2}+J(\bxi)[u],
\ \ u\in L_{2}(\Omega),
\end{multline}
where the functional $J(\bxi)[u]$ is defined by
\begin{equation}
\label{ue1.15}
J(\bxi)[u]:=\intop_{\Omega}d\x\,q_{0}(\x)\intop_{\R^d}d\y\, a(\x-\y)\mu(\x,\y)e^{-i \langle \bxi, \y \rangle}\overline{u(\y)}
\bigl(e^{i \langle \bxi, \x \rangle}u(\x)-e^{i \langle \bxi, \y\rangle}u(\y) \bigr),\ \ u\in L_{2}(\Omega).
\end{equation}
Rearranging the expression in the middle part of \eqref{ue1.14} we obtain
\begin{equation}\label{ue1.26a}
([q_{0}]\A(\bxi; a,\mu)u,u)
=\intop_{\Omega}d\x\,q_{0}(\x)p(\x)|u(\x)|^{2}-\intop_{\Omega}d\x\,q_{0}(\x)\intop_{\R^{d}}d\y\, a(\x-\y)\mu(\x,\y)e^{-i\langle\bxi,\x-\y\rangle}\overline{u(\x)}u(\y).
\end{equation}
Dividing the right-hand side of  (\ref{ue1.15}) into two terms and  using \eqref{10.02.1} and \eqref{ue1.26a}, we have
\begin{multline}\label{ue1.16}
J(\bxi)[u]=-\intop_{\Omega}d\x\, q_{0}(\x) \intop_{\R^d}d\y\, a(\x-\y)\mu(\x,\y)|u(\y)|^{2}
\\
+\intop_{\Omega}d\x\,q_{0}(\x)\intop_{\R^{d}}d\y\, a(\x-\y)\mu(\x,\y)e^{i\langle\bxi,\x-\y\rangle}u(\x)\overline{u(\y)}
\\
=-\intop_{\Omega}d\y\,|u(\y)|^{2}\intop_{\Omega}d\x\, q_{0}(\x)\widetilde a(0,\x-\y)\mu(\x,\y)
\\
+\intop_{\Omega}d\x\,q_{0}(\x)\intop_{\R^{d}}d\y\, a(\x-\y)\mu(\x,\y)e^{i\langle\bxi,\x-\y\rangle}u(\x)\overline{u(\y)}
\\
=-\intop_{\Omega}d\y\,|u(\y)|^{2}q_{0}(\y)p(\y)
\\
+\intop_{\Omega}d\x\,q_{0}(\x)\intop_{\R^{d}}d\y\, a(\x-\y)\mu(\x,\y)e^{i\langle\bxi,\x-\y\rangle}u(\x)\overline{u(\y)}
\\
=-\overline{( [q_0] \A(\bxi; a,\mu)u,u)},\ \ u\in L_{2}(\Omega).
\end{multline}
FInally, relation (\ref{ue1.11}) follows from  (\ref{ue1.14}) and (\ref{ue1.16}).
\end{proof}

\begin{corollary}
The operator $[q_{0}^{1/2}]\A(\bxi; a,\mu) [q_{0}^{-1/2}]$ is accretive, the spectrum $\sigma(\A(\bxi; a,\mu))$ is
situated  in the right half-plane $\{z\in\mathbb C\,:\, \mathrm{Re}\, z\geqslant 0\}$; the following estimate holds:
\begin{equation*}\label{10.02.2.1}
\|(\A(\bxi; a,\mu)-zI)^{-1}\|\le q_{+}^{1/2}q_{-}^{-1/2}|\operatorname{Re}z|^{-1},\ \ \operatorname{Re}z<0,\ \ \bxi\in\widetilde\Omega.
\end{equation*}
\end{corollary}

\begin{proof}
The proof of this statement is similar to that of Corollary \ref{corollary1.7}.
\end{proof}

\subsection{Estimates of the quadratic form of the operator $\operatorname{Re}[q_{0}]\A(\bxi;a,\mu)$}\label{sec1.4}
We consider separately the case $\mu=\mu_{0}\equiv 1$.  For brevity, let us denote
$\A_{0}(\bxi; a):=\A(\bxi; a,\mu_{0})$, $\B_{0}(\bxi; a):=\B(\bxi; a,\mu_{0})$.
Remark \ref{remark1.5} and relations (\ref{h1.2}), (\ref{d1.0.0}) and (\ref{ue1.11}) imply the following two-sided
estimates:
\begin{multline}
\label{e1.17}
\mu_{-}q_{-}\mathrm{Re}(\A_{0}(\bxi; a)u,u)\le\mathrm{Re}([q_{0}]\A(\bxi; a,\mu)u,u)\le\mu_{+}q_{+}\mathrm{Re}(\A_{0}(\bxi; a)u,u),
\\
u\in L_{2}(\Omega),\ \ \bxi\in\widetilde\Omega.
\end{multline}
The operators $\A_{0}(\bxi; a)$, $\bxi\in\widetilde\Omega$,  are diagonalized by means of the unitary discrete
Fourier transform ${\mathcal F}: L_{2}(\Omega)\to \ell_{2}(\Z^d)$ defined by the relations
\begin{gather*}
{\mathcal F} u(\n)=\intop_{\Omega}u(\x)e^{-2\pi i \langle \n, \x \rangle}d\x,\ \ \n\in\Z^d,\ \ u\in L_{2}(\Omega);
\\
{\mathcal F}^{*}v(\x)=\sum_{\n\in\Z^d}v_{\n}e^{2\pi i \langle \n, \x\rangle},\ \ \x\in\Omega,\ \
v=\{v_{\n}\}_{\n\in\Z^d}\in\ell_{2}(\Z^d).
\end{gather*}
By direct computation it is verified that
\begin{equation}
\label{e1.18}
\A_{0}(\bxi; a)= {\mathcal F}^{*} \bigl[\widehat a({\mathbf 0})-\widehat a(2\pi \n+\bxi) \bigr] {\mathcal F},\ \ \widehat a(\k):=\intop_{\R^d}a(\x)e^{-i \langle \k, \x \rangle}d\x,\ \ \k\in\R^d,
\end{equation}
where the notation $[\widehat{a}(\mathbf{0})-\widehat a(2\pi \n+\bxi)]$  stands for operator of multiplication
by the function  \hbox{$\widehat{a}(\mathbf{0})-\widehat a(2\pi \n+\bxi)$} in the space  $\ell_2(\Z^d)$.
Thus, the sequence  $\{\widehat A_{\n}(\bxi)\}_{\n\in\Z^d}$ is the symbol of the operator $\operatorname{Re}\A_{0}(\bxi; a)$ with
\begin{equation*}
\widehat A_{\n}(\bxi)=\widehat A(\bxi+2\pi \n)=\operatorname{Re}(\widehat a(\mathbf{0})-\widehat a(2\pi \n+\bxi))
=\intop_{\R^d} \bigl(1-\cos( \langle \z, \bxi+2\pi \n \rangle) \bigr)a(\z)\,d\z, \ \ \n\in\Z^d,\ \ \bxi\in\widetilde\Omega.
\end{equation*}
Let us explore the symbol of the operator $\operatorname{Re}\A_{0}(\bxi; a)$, $\bxi\in\widetilde\Omega$, in more detail
and estimate the quadratic form of the operator $\operatorname{Re}\A(\bxi; a,\mu)$. To this end notice that, under
condition \eqref{h1.1},  the quantity
\begin{equation*}
\label{e1.17a}
\widehat A(\y) :=\intop_{\R^d} \bigl(1-\cos(\langle \z, \y\rangle) \bigr)a(\z)\,d\z
= 2\intop_{\R^d}\sin^{2}\left(\frac{\langle \z, \y\rangle}{2}\right)a(\z)\,d\z,\ \ \y \in \R^d,
\end{equation*}
depends continuously on $\y\in\R^d$ and, by the Riemann-Lebesgue Lemma, converges to $\|a\|_{L_1}>0$, as
$| \y|\to\infty$. It is also clear that $\widehat A(\y) >0$, if $\y\not=0$.
Consequently,
\begin{equation}
\label{e1.19}
\min_{| \y| \ge r} \widehat A(\y)=:\mathcal{C}_{r}(a)>0,\ \ r>0.
\end{equation}
Since for $\bxi\in\widetilde\Omega$ and $\n\in\Z^d\setminus\{ \mathbf{0} \}$ we have $| \bxi + 2\pi \n |\ge \pi$, then
\begin{equation}
\label{e1.22}
\widehat A(\bxi+2\pi \n)\ge \mathcal{C}_{\pi}(a),\ \ \bxi \in \widetilde\Omega,\ \ \n\in\Z^d\setminus\{\mathbf{0}\}.
\end{equation}

Under the condition $M_2(a) < \infty$ the function
\begin{equation*}
\label{e1.18a}
\intop_{\R^d}a(\z) \langle \z, \y \rangle^{2}\, d\z=:M_{a}(\y)
\end{equation*}
depends continuously on $\y\in\R^d$, and $M_{a}(\y)>0$, if $\y\not=0$. Therefore,
\begin{equation}
\label{e1.20}
\min_{|\boldsymbol{\theta}|=1}M_{a}(\boldsymbol{\theta}) =:\mathcal{M}(a)>0.
\end{equation}

\begin{lemma}
\label{lemma1.3}
Let conditions  \eqref{h1.1}--\eqref{h1.3} be fulfilled, and assume that $M_2(a)<\infty$. Then
the following estimate holds:
\begin{equation*}
\label{e1.21}
\widehat A(\bxi+2\pi \n)\ge C(a)|\bxi|^{2},\ \ \bxi\in\widetilde\Omega,\ \ \n\in\Z^d,
\end{equation*}
where
\begin{equation}
\label{C(a)}
C(a):=\min \Bigl\{\frac{1}{4}\mathcal{M}(a),\mathcal{C}_{r(a)}(a)\pi^{-2}d^{-1},\mathcal{C}_{\pi}(a)\pi^{-2}d^{-1}
\Bigr\}>0,
\end{equation}
the value of $\mathcal{C}_{r}(a)$, $r>0$, is given by  \eqref{e1.19}, the constant $\mathcal{M}(a)$ is defined in
\eqref{e1.20}, and $r(a)$ satisfies condition  \eqref{dop5} below.
\end{lemma}
\begin{proof}
Define the function $\Phi(\lambda)$ by the relation
\begin{equation}\label{dop1}
1-\cos\lambda=\lambda^{2}\int_{0}^{1}dt\,t\,\int_{0}^{1}ds\,\cos(st\lambda)=:\lambda^{2}\Phi(\lambda),\ \ \lambda\in\mathbb{R}.
\end{equation}
Then $\Phi(\lambda)$ possesses the following properties:
\begin{equation}\label{dop2}
|\Phi(\lambda)|\le 1/2,\ \ \hbox{and } \ \Phi(\lambda)\to 1/2,\ \ \hbox{as }\lambda\to 0.
\end{equation}
By \eqref{dop1} and the definition of $\widehat A$ and $M_{a}$ we have
\begin{equation}\label{dop3}
\Bigl| \widehat A(\bxi)-\frac{1}{2}M_{a}(\bxi) \Bigr| \le |\bxi|^{2}\int_{\mathbb{R}^{d}}a(\z)|\z|^{2} \Bigl|\Phi(\langle\bxi,\z\rangle)-\frac{1}{2}\Bigr| \, d\z=:|\bxi|^{2}\mathcal{E}(\bxi),\ \ \bxi\in\mathbb{R}^{d}.
\end{equation}
In view of  \eqref{dop2}  and the condition $M_{2}(a)<\infty$, the function $\mathcal{E}(\bxi)$ tends to zero
as $|\bxi|\to 0$.
We choose $r(a)>0$ in such a way that
\begin{equation}\label{dop5}
\mathcal{E}(\bxi)\le\frac{1}{4}\mathcal{M}(a),\ \ |\bxi|\le r(a).
\end{equation}
Combining  \eqref{e1.20}, \eqref{dop3} and \eqref{dop5} yields
\begin{equation}\label{dop6}
\widehat A(\bxi)\ge\frac{1}{4}M_{a}(\bxi)\ge\frac{1}{4}\mathcal{M}(a)|\bxi|^{2},\ \ |\bxi|\le r(a),
\end{equation}
and the statements of Lemma follow from \eqref{e1.19}, \eqref{e1.22} and \eqref{dop6}.
\end{proof}

\begin{remark}
It is straightforward to check that, under the additional condition $M_{3}(a)<\infty$, the radius $r(a)$
can be calculated explicitly in terms of $\mathcal{M}(a)$ and $M_{3}(a)$, see  \cite[Lemma 1.3]{PSlSuZh}.
\end{remark}

Let us  return to the study of the operator  $\A(\bxi;a,\mu)$ in the general case.
Under conditions   \eqref{h1.1}--\eqref{h1.3}, relations \eqref{e1.17}, \eqref{e1.18} and \eqref{e1.22}
yield the estimate
\begin{equation*}\label{e1.31}
\mathrm{Re}([q_{0}]\A(\bxi; a,\mu)u,u)\ge \mu_{-}q_{-}\mathcal{C}_{\pi}(a)\|u\|^{2}_{L_2(\Omega)},\ \
u\in L_{2}(\Omega)\ominus\mathcal{L}\{\mathbf{1}_{\Omega}\},\ \
\bxi\in\widetilde\Omega.
\end{equation*}
The following statement is a consequence of  (\ref{e1.17}), (\ref{e1.18}) and Lemma \ref{lemma1.3}:

\begin{proposition}
\label{prop1.4}
Let conditions \eqref{h1.1}--\eqref{h1.3} be satisfied, and assume that $M_2(a)<\infty$. Then
\begin{equation}\label{e1.30}
\mathrm{Re}([q_{0}]\A(\bxi;a,\mu)u,u)\ge \mu_{-}q_{-}C(a)|\bxi|^{2}\|u\|_{L_2(\Omega)}^{2},\ \ u\in
L_{2}(\Omega),\ \ \bxi\in\widetilde\Omega.
\end{equation}
\end{proposition}

\begin{remark}
In the symmetric case a similar estimate was proved in \cite[Proposition 1.4]{PSlSuZh}.
\end{remark}

\begin{corollary}\label{corollary1.33}
For any $\bxi\in\widetilde\Omega$ the operator
$[q_{0}^{1/2}] \A(\bxi; a,\mu) [q_{0}^{-1/2}] -\mu_{-}q_{-}q_{+}^{-1}C(a)|\bxi|^{2}I$ is accretive,
the spectrum of
$\A(\bxi; a,\mu)$ is situated in the half-plane $\{z\in\mathbb{C}:
\operatorname{Re} z\geqslant \mu_{-}q_{-}q_{+}^{-1}C(a)|\bxi|^{2}\}$,
and the following estimate holds:
\begin{equation}\label{e1.30.1}
\bigl\|(\A(\bxi; a,\mu)-zI)^{-1} \bigr\| \le q_{-}^{-1/2}q_{+}^{1/2} \bigl(\mu_{-}q_{-}q_{+}^{-1}C(a)|\bxi|^{2}-
\operatorname{Re} z \bigr)^{-1},\ \ \operatorname{Re} z < \mu_{-}q_{-}q_{+}^{-1}C(a)|\bxi|^{2}.
\end{equation}
\end{corollary}

\begin{proof}
By virtue of \eqref{e1.30} we have
\begin{multline*}
\operatorname{Re} (q_{0}^{1/2}\A(\bxi; a,\mu)q_{0}^{-1/2}u,u)\geqslant \mu_{-}q_{-}C(a)|\bxi|^{2}\|q_{0}^{-1/2}u\|^{2}
\\
\ge\mu_{-}q_{-}q_{+}^{-1}C(a)|\bxi|^{2}\|u\|^{2},\ \ u\in L_{2}(\Omega),\ \ \bxi\in\widetilde\Omega.
\end{multline*}
Therefore, the operator
$[q_{0}^{1/2}] \A(\bxi; a,\mu) [q_{0}^{-1/2}] - \mu_{-}q_{-}q_{+}^{-1}C(a)|\bxi|^{2}I$
is accretive,  for any $z$ such that  $\operatorname{Re} z<0$ the operator
$[q_{0}^{1/2}] \A(\bxi; a,\mu) [q_{0}^{-1/2}] - \mu_{-}q_{-}q_{+}^{-1}C(a)|\bxi|^{2}I-zI$
is invertible, and the inverse operator satisfies the estimate
\begin{equation}\label{e1.30.2}
\bigl\| \bigl( [q_{0}^{1/2}] \A(\bxi; a,\mu) [q_{0}^{-1/2}] - \mu_{-}q_{-}q_{+}^{-1}C(a)|\bxi|^{2}I-zI \bigr)^{-1} \bigr\|\le
| \operatorname{Re} z|^{-1}.
\end{equation}
Now all the statements of Corollary follow from \eqref{e1.30.2}.
\end{proof}

We will also need the estimate
\begin{equation}
\label{e1.32}
\|\A(\bxi;a,\mu)-\A(\boldsymbol{\eta};a,\mu)\|\le\mu_{+} M_{1}(a)|\bxi- \boldsymbol{\eta} |,\ \
\bxi, \boldsymbol{\eta} \in\widetilde\Omega,
\end{equation}
which is valid under conditions \eqref{h1.1}--\eqref{h1.3} and $M_1(a) < \infty$.
This estimate can be deduced  from \eqref{A_xi}--\eqref{a_tilde} by means of the Schur lemma.

\section{Threshold characteristics of convolution type operators   \\   near zero}

\subsection{The spectrum of operator  $\A(\bxi;a,\mu)$ near zero}
According to Proposition  \ref{YadroA}, under conditions (\ref{h1.1})--(\ref{h1.3}), the point
$\lambda_{0}=0$ is an isolated eigenvalue of the operator $\A(\mathbf{0};a,\mu)$, and the corresponding
Riesz projector has rank $1$. For brevity we denote by $d_{0}:=d_{0}(a,\mu)$ the distance from  $\lambda_{0}$ to the remaining part of the spectrum of  $\A(\mathbf{0};a,\mu)$.
The resolvent of the operator $\A(\mathbf{0};a,\mu)$ is holomorphic in the disk $B_{d_{0}}(0)$ with the point $0$  punctured. Therefore, the quantity
\begin{equation}
\label{const=K}
K:=K(d_{0},a,\mu):=\max_{\frac{d_{0}}{3}\le|\zeta|\le\frac{2d_{0}}{3}}\|(\A(\mathbf{0};a,\mu)-\zeta I)^{-1}\|
\end{equation}
is finite. Denote
\begin{equation}
\label{delta0}
\delta_{0}(a,\mu):= \min\left\{ \frac{\pi}{2},\frac{1}{(d_{0} K^{2}+3K ) M_{1}(a) \mu_{+}} \right\}.
\end{equation}
Clearly the ball $\{\bxi\in\mathbb R^d\,:\,|\bxi| \leqslant \delta_{0}(a,\mu)\}$ is a subset of $\widetilde\Omega$.

The following statement is a straightforward  consequence of estimate (\ref{e1.32}), Proposition \ref{rizprp3} and  Remark \ref{rizremark}:
\begin{proposition}
\label{prop2.1}
Let conditions \eqref{h1.1}--\eqref{h1.3}  be fulfilled, and assume that $M_1(a) < \infty$. Then for all $\bxi$ such that $|\bxi|\le\delta_{0}(a,\mu)$
the spectrum of the operator $\A(\bxi) = \A(\bxi;a,\mu)$ in the disc $B_{d_0/3}(0)$ consists of one simple eigenvalue,
i.e. the Riesz projector corresponding to the spectrum of $\A(\bxi)$ in the disc $B_{d_0/3}(0)$ has rank $1$.
The set  $\{\zeta\in\mathbb{C}:\frac{d_{0}}{3}\le|\zeta|\le\frac{2d_{0}}{3}\}$ does not intersect with  $\sigma(\A(\bxi))$.
\end{proposition}

Under conditions \eqref{h1.1}--\eqref{h1.3}, the regularity of the operator-function $\A(\cdot;a,\mu)$ depends on how many moments of the function $a(\cdot)$ are finite. If $M_k(a) < \infty$, then $\A(\cdot;a,\mu)$ is $k$ times  continuously differentiable in the norm of $\mathcal{B}(L_{2}(\Omega))$, and the derivatives of  this operator-function are given by
%
\begin{equation}
\label{e2.1}
\begin{aligned}
\partial^{\alpha}\A(\bxi; a,\mu)u(\x) &=\intop_{\Omega}\widetilde
a_{\alpha}(\bxi,\x-\y)\mu(\x,\y)u(\y)\, d\y,\ \ \x\in\Omega,\ \ u\in L_{2}(\Omega);
\\
\widetilde a_{\alpha}(\bxi,\z) &=(-1)(-i)^{|\alpha|}\sum_{\n\in\Z^d}(\z+\n)^{\alpha}a(\z+\n)e^{-i \langle \bxi, \z+\n \rangle},\
\ \alpha\in\Z^d_{+},\ \ |\alpha|\le k.
\end{aligned}
\end{equation}

\begin{lemma}
Let conditions \eqref{h1.1}--\eqref{h1.3} hold true. Then

\noindent
$1^\circ$. If $M_1(a) < \infty$, then
\begin{equation}
\label{2.3}
\| \A(\bxi) - \A(\mathbf{0})\| \le \mu_+ M_1(a) |\bxi|,\quad |\bxi| \le \delta_0(a,\mu).
\end{equation}

\noindent
$2^\circ$. If $M_2(a) < \infty$, then
\begin{align}
\label{2.4}
\A(\bxi)&=\A(\mathbf{0})+ [\Delta_1\A](\bxi ) +
\mathbb{K}_1(\bxi),\quad  [\Delta_1\A](\bxi ) :=\sum_{j=1}^{d}\partial_{j}\A(\mathbf{0})\xi_{j},
\\
\label{e2.3}
\left\| [\Delta_1\A](\bxi ) \right\| &\le \mu_{+} M_{1}(a) |\bxi|,\ \ |\bxi| \le \delta_0(a,\mu),\\
\label{2.5}
\| \mathbb{K}_1(\bxi) \| &\le \frac{1}{2} \mu_+ M_2(a) |\bxi|^2,\quad |\bxi| \le \delta_0(a,\mu).
\end{align}

\noindent
$3^\circ$. If $M_3(a) < \infty$, then
\begin{align}
\label{2.6}
\A(\bxi)&=\A(\mathbf{0})+ [\Delta_1\A](\bxi ) + [\Delta_2\A](\bxi ) + \mathbb{K}_2(\bxi),\quad [\Delta_2\A](\bxi ) := \frac{1}{2}\sum_{k,l=1}^{d}\partial_{k}\partial_{l}\A(\mathbf{0})\xi_{k}\xi_{l},
\\
\label{e2.4}
\left\| [\Delta_2\A](\bxi ) \right\|&
\le \frac{1}{2}\mu_{+} M_{2}(a) |\bxi|^2,\ \ |\bxi| \le \delta_0(a,\mu),
\\
\label{2.7}
\| \mathbb{K}_2(\bxi) \| &\le \frac{1}{6} \mu_+ M_3(a) |\bxi|^3,\quad |\bxi| \le \delta_0(a,\mu).
\end{align}
\end{lemma}

Under additional symmetry assumptions $\mu(\x,\y)=\mu(\y,\x)$, $\x,\y\in\mathbb{R}^{d}$ and $a(\z)=a(-\z)$, $\z\in\mathbb{R}^{d}$, a similar result was  obtained in \cite[Lemma 2.2]{PSlSuZh1}. However, the symmetry conditions
were not used in the proof, and thus the proof remains valid in the non-symmetric case studied here.

\subsection{Threshold approximations}
Denote by $F(\bxi)$ the Riesz projector of the operator $\A(\bxi) = \A(\bxi;a,\mu)$ that corresponds to the disc $\{z\in\mathbb{C}:|z|\le d_{0}/3\}$.
We recall that the symbol $P$ stands for the Riesz projector of the operator $\A(\mathbf{0};a,\mu)$ that corresponds to the eigenvalue
$\lambda_{0}=0$. Then we have $P = (\cdot, q_{0})\1_\Omega$.
Denote by $\Gamma$  the circle  $\{z\in\mathbb{C}:|z|= d_{0}/2\}$ which encloses the disc $\{z\in\mathbb{C}:|z|\le d_{0}/3\}$
equidistantly.
By the Riesz formula, $F(\bxi)$ and  $\A(\bxi) F(\bxi)$ admit the representations
\begin{align}
\label{e2.6}
F(\bxi)  &= - \frac{1}{2\pi i}\ointop_{\Gamma}(\A(\bxi)-\zeta I)^{-1}\, d\zeta,\ \ |\bxi|\le\delta_{0}(a,\mu),
\\
\label{e2.6a}
\A(\bxi) F(\bxi) &= - \frac{1}{2\pi i}\ointop_{\Gamma}(\A(\bxi)-\zeta I)^{-1}\zeta\,d\zeta,\ \ |\bxi|\le\delta_{0}(a,\mu);
\end{align}
here the integration along the contour is performed in a counterclockwise direction.
Our goal is to approximate the operator $F(\bxi)$ for small $|\bxi|$ with accuracy $O(|\bxi|)$,
and the operator $\A(\bxi)F(\bxi)$ with accuracy  $O(|\bxi|^3)$.
In the symmetric case such approximations were constructed in the work \cite{PSlSuZh}.
We essentially use the technique developed in this work; to construct the desired approximations, we integrate
the resolvent of   $\A(\bxi)$ along a proper contour (this method is called the "third method" \ in  \cite[\S~4]{PSlSuZh}).
\begin{proposition}
\label{prop2.3}
Let conditions \eqref{h1.1}--\eqref{h1.3} be satisfied, and assume that $M_1(a) < \infty$. Then
\begin{equation}
\label{F-P}
\| F(\bxi) - P \| \le C_1(a,\mu) |\bxi|, \quad  |\bxi| \le \delta_0(a,\mu);
\end{equation}
the constant $C_1(a,\mu)$ is defined in   \eqref{C1} below, it depends on $\mu_+$,  $M_1(a)$, $d_{0}$ and $K$.
\end{proposition}

\begin{proof}
For the sake of brevity we introduce the notation
\begin{align*}
R(\bxi,\zeta)&:=(\A(\bxi)-\zeta I)^{-1},\ \
|\bxi|\le\delta_{0}(a,\mu),\ \ \zeta\in\Gamma;
\\
R_{0}(\zeta)&:=R(\mathbf{0},\zeta),\ \ \zeta\in\Gamma;
\\
\Delta\A(\bxi)&:=\A(\bxi)-\A(\mathbf{0}),\ \
|\bxi|\le\delta_{0}(a,\mu).
\end{align*}
Due to the Riesz formula in \eqref{e2.6}, the difference  $F(\bxi)-P$ admits the representation
\begin{equation}\label{e2.30}
F(\bxi)-P= - \frac{1}{2\pi i}\ointop_{\Gamma} \left( R(\bxi,\zeta) - R_0(\zeta)\right) \, d\zeta,
\ \ |\bxi|\le\delta_{0}(a,\mu).
\end{equation}
In what follows we systematically use the resolvent identity
\begin{equation}\label{e2.31}
R(\bxi,\zeta) = R_0(\zeta)
- R(\bxi,\zeta) \Delta\A(\bxi) R_0(\zeta), \quad
|\bxi|\le\delta_{0}(a,\mu),\ \ \zeta\in\Gamma,
\end{equation}
and its iterations. Notice that the length of the contour $\Gamma$ is equal to $\pi d_{0}$.
Let us show that  the resolvents  $R_0 (\zeta)$ and $R(\bxi,\zeta)$ can be estimated on the contour
$\Gamma$ as follows:
\begin{equation}\label{e2.32}
\| R_0 (\zeta)\|\leqslant K,\ \
\| R(\bxi,\zeta)\|\leqslant 3K /2, \ \ |\bxi|\leqslant\delta_{0}(a,\mu),\quad  \zeta \in \Gamma,
\end{equation}
where the constant $K$ is defined in  \eqref{const=K}.
The former inequality is an immediate consequence of \eqref{const=K}. In order to derive the latter one
one should take into account estimate  \eqref{2.3},  relation  \eqref{delta0} and the fact that
the resolvent $R(\bxi,\zeta)$ is represented on the contour $\Gamma$ by the convergent series
$$
R(\bxi,\zeta)=R_0 (\zeta)\sum_{n=0}^{\infty}\left(-\Delta\A(\bxi)R_0 (\zeta)\right)^{n},\ \ |\bxi|\le\delta_{0}(a,\mu),\quad \zeta \in \Gamma.
$$
Finally, combining \eqref{2.3}  and \eqref{e2.30}--\eqref{e2.32} yields the desired estimate \eqref{F-P}
with the constant
\begin{equation}
\label{C1}
C_{1}(a,\mu):= \frac{3}{4} K^{2}d_{0}\mu_{+}M_{1}(a).
\end{equation}
\end{proof}

We turn to threshold approximations for the operator  $\A(\bxi)F(\bxi)$. Here the technique differs essentially from
that exploited in the self-adjoint case.

\begin{proposition}
\label{prop2.6}
Let conditions \eqref{h1.1}--\eqref{h1.3} be fulfilled, and assume that $M_3(a) < \infty$. Then
the operator-function $ \A(\bxi)F(\bxi)$ admits the expansion
\begin{equation}
\label{AF = G2}
 \A(\bxi)F(\bxi) =  [G]_{1}(\bxi)+[G]_2(\bxi) + \Psi(\bxi), \quad [G]_{1}(\bxi):=\sum_{j=1}^{d}G_{j}\xi_{j}, \quad [G]_2(\bxi) := \frac{1}{2}\sum_{k,l=1}^d G_{kl} \xi_k \xi_l,
\end{equation}
and the remainder $ \Psi(\bxi)$ satisfies the estimate
\begin{equation}
\label{AF-G}
\left\| \Psi (\bxi)  \right\| \le C_2(a,\mu) |\bxi|^3, \quad  |\bxi| \le \delta_0(a,\mu).
\end{equation}
The operators $G_{j}$ are defined by
\begin{equation}\label{G_{j}=prop}
G_{j}=P \partial_{j}\mathbb{A}(\mathbf{0}) P,\ \ j=1,\dots,d,
\end{equation}
and the operators  $G_{kl}$ have the form
\begin{equation}
\label{G_kl=prop}
\begin{split}
G_{kl} =&  P \partial_{k} \partial_{l} \mathbb{A}(\mathbf{0}) P
- P \partial_{k}  \mathbb{A}(\mathbf{0}) R_{1}(0)  \partial_{l}  \mathbb{A}(\mathbf{0}) P
-  P \partial_{l}  \mathbb{A}(\mathbf{0}) R_{1}(0)  \partial_{k}  \mathbb{A}(\mathbf{0}) P
\\
&- P \partial_{k}  \mathbb{A}(\mathbf{0})P\partial_{l}  \mathbb{A}(\mathbf{0})R_{1}(0)
- R_{1}(0)\partial_{k}  \mathbb{A}(\mathbf{0})P\partial_{l}  \mathbb{A}(\mathbf{0})P
\\
&- P \partial_{l}  \mathbb{A}(\mathbf{0})P\partial_{k}  \mathbb{A}(\mathbf{0})R_{1}(0)
- R_{1}(0)\partial_{l}  \mathbb{A}(\mathbf{0})P\partial_{k}  \mathbb{A} (\mathbf{0}) P,\ \ k,l=1,\dots,d;
\end{split}
\end{equation}
here $R_{1}(0):=Q\mathbb{A}(\mathbf{0})^{-1}Q$, $Q=I-P$, and   $\mathbb{A}(\mathbf{0})^{-1}$
is treated as the operator inverse to  $ \mathbb{A}(\mathbf{0})
\vert_{QL_{2}(\Omega)} : {QL_{2}(\Omega)} \to {QL_{2}(\Omega)}$.
This operator is well-defined and bounded.
The constant $C_2(a,\mu)$ depends on $d_{0}$, $K$, $\mu_+$,  $M_1(a)$, $M_2(a)$ and $M_3(a)$; its value is specified
in formula \eqref{C3} below.
\end{proposition}

\begin{proof}
Iterating the resolvent identity in \eqref{e2.31} two times we obtain
\begin{equation}
\label{res_iden3}
\begin{aligned}
R(\boldsymbol{\xi},\zeta) =& \,R_0(\zeta)
- R_0(\zeta) \Delta\A(\boldsymbol{\xi}) R_0(\zeta) + R_0(\zeta) \Delta\A(\boldsymbol{\xi}) R_0(\zeta) \Delta\A(\boldsymbol{\xi}) R_0(\zeta)
+ Z_3(\boldsymbol{\xi},\zeta),
\\
Z_3(\boldsymbol{\xi},\zeta) :=&
- R(\boldsymbol{\xi},\zeta) \Delta\A(\boldsymbol{\xi}) R_0(\zeta) \Delta\A(\boldsymbol{\xi}) R_0(\zeta) \Delta\A(\boldsymbol{\xi}) R_0(\zeta), \quad
| \boldsymbol{\xi} |\le\delta_{0}(a,\mu),\ \ \zeta\in\Gamma.
\end{aligned}
\end{equation}
By \eqref{2.3} and \eqref{e2.32}, $Z_3( \boldsymbol{\xi},\zeta)$ satisfies the estimate
\begin{equation}
\label{res_est4a}
\| Z_3( \boldsymbol{\xi},\zeta) \| \le \frac{3}{2}K^{4} \mu^3_+ M_1(a)^3 |\bxi|^3,
\quad
| \boldsymbol{\xi} | \leqslant \delta_{0}(a,\mu),\ \ \zeta\in\Gamma.
\end{equation}
Next,  substituting   \eqref{2.6}  into the second term on the right-hand side of \eqref{res_iden3}, and  \eqref{2.4} -- into the third term, we arrive at the relations
\begin{equation}
\label{res_iden4}
\begin{aligned}
R(\bxi,\zeta) =& \,R_0(\zeta)
- R_0(\zeta) \left( [\Delta_1\A](\bxi ) + [\Delta_2\A](\bxi )  \right)  R_0(\zeta)
\\
+& R_0(\zeta)   [\Delta_1\A](\bxi )   R_0(\zeta)
 [\Delta_1\A](\bxi )  R_0(\zeta)
+ Z_3(\bxi,\zeta) + Z_4(\boldsymbol{\xi},\zeta),
\\
Z_4(\bxi,\zeta) :=&
- R_0(\zeta) \mathbb{K}_2(\bxi)  R_0(\zeta) + R_0(\zeta) \mathbb{K}_1(\bxi)  R_0(\zeta) \Delta\A(\bxi) R_0(\zeta)
\\
+& R_0(\zeta) [\Delta_1\A](\bxi )   R_0(\zeta) \mathbb{K}_1(\bxi)  R_0(\zeta)
, \quad
|\bxi|\le\delta_{0}(a,\mu),\ \ \zeta\in\Gamma.
\end{aligned}
\end{equation}
Due to  \eqref{2.3},  \eqref{e2.3},  \eqref{2.5}, \eqref{2.7} and \eqref{e2.32} the term $Z_4(\bxi,\zeta)$ satisfies the estimate
 \begin{equation}
\label{res_est5}
\| Z_4(\bxi,\zeta) \| \le \left(\frac{1}{6}\mu_+ K^{2}M_3(a)  +  \mu_+^2 K^{3}M_1(a)M_2(a) \right)|\bxi|^3 ,
\quad
|\bxi|\le\delta_{0}(a,\mu),\ \ \zeta\in\Gamma.
\end{equation}
Using Riesz formula \eqref{e2.6a} and representation \eqref{res_iden4} we conclude that
\begin{equation}
\label{AF=}
\A(\bxi)F(\bxi) = G_0  + \sum_{j=1}^d G_j \xi_j + \frac{1}{2} \sum_{k,l=1}^d G_{kl} \xi_k \xi_l + \Psi(\bxi), \quad
|\bxi|\le\delta_{0}(a,\mu),
\end{equation}
with
\begin{align}
\label{G0=}
G_0 &= - \frac{1}{2\pi i}\oint_{\Gamma} R_0(\zeta) \zeta  \, d\zeta,
\\
\label{Gj=}
G_j &=  \frac{1}{2\pi i}\oint_{\Gamma} R_0(\zeta)\partial_{j}\A(\mathbf{0}) R_0(\zeta) \zeta  \, d\zeta,
\\
\label{Gkl=}
\begin{split}
G_{kl} &=  \frac{1}{2\pi i}\oint_{\Gamma} R_0(\zeta)\partial_{k} \partial_{l} \A(\mathbf{0}) R_0(\zeta) \zeta  \, d\zeta
\\
&- \frac{1}{2\pi i}\oint_{\Gamma}  \left( R_0(\zeta)\partial_{k} \A(\mathbf{0}) R_0(\zeta) \partial_{l}
\A(\mathbf{0})R_0(\zeta) +
R_0(\zeta)\partial_{l} \A(\mathbf{0}) R_0(\zeta) \partial_{k} \A(\mathbf{0})R_0(\zeta) \right) \zeta  \, d\zeta,
\end{split}
\\
\label{Psi=}
\Psi(\bxi) &= - \frac{1}{2\pi i}\oint_{\Gamma} (Z_3(\bxi,\zeta) + Z_4(\bxi,\zeta)) \zeta  \, d\zeta.
\end{align}
Combining \eqref{res_est4a}, \eqref{res_est5} and \eqref{Psi=} we obtain estimate  \eqref{AF-G} with the constant
\begin{equation}
\label{C3}
C_{2}(a,\mu):=\frac{d^2_{0}}{4} \left( \frac{3}{2}K^{4} \mu_{+}^3 M_{1}(a)^3 + \frac{1}{6}K^{2}\mu_{+} M_{3}(a)
+ K^{3}\mu_{+}^2 M_{1}(a) M_2(a) \right).
\end{equation}

In order to calculate the integrals in \eqref{G0=}--\eqref{Gkl=} we use the following decomposition of the resolvent
of $\mathbb{A}(\mathbf{0})$:
\begin{equation}\label{4.10}
R_{0}(\zeta)=R_{0}(\zeta)P+R_{0}(\zeta)(I-P)=-\frac{1}{\zeta}P+R_{0}(\zeta)(I-P),\ \ \zeta\in\Gamma.
\end{equation}

It is straightforward to check that
\begin{equation}\label{R1}
R_{0}(\lambda)(I-P)=\frac{1}{2\pi i}\oint_{\Gamma}(\zeta-\lambda)^{-1}R_{0}(\zeta)\,d\zeta,\ \ 0<|\lambda|<d_{0}/2.
\end{equation}
Indeed, with the help of  the Hilbert identity $R_{0}(\zeta) = R_{0}(\lambda) + (\zeta - \lambda) R_{0}(\lambda)R_{0}(\zeta)$,
the right-hand side of  \eqref{R1} can be rearranged as follows:
$$
\frac{1}{2\pi i}\oint_{\Gamma}(\zeta-\lambda)^{-1}R_{0}(\zeta)\,d\zeta =
R_0(\lambda) \frac{1}{2\pi i}\oint_{\Gamma}(\zeta-\lambda)^{-1}\,d\zeta
+R_0(\lambda) \frac{1}{2\pi i}\oint_{\Gamma} R_0(\zeta)\,d\zeta = R_0(\lambda) - R_0(\lambda)P;
$$
here we have also used  \eqref{e2.6} with $\bxi= {\mathbf 0}$.

In view of representation \eqref{R1}, the operator-function $R_{1}(\lambda):=R_{0}(\lambda)(I-P)$, $0<|\lambda|\leqslant d_{0}/2$, can be extended to a function  which is holomorphic inside the contour $\Gamma$,
the operator $ \mathbb{A}(\mathbf{0})
\vert_{QL_{2}(\Omega)} : {QL_{2}(\Omega)} \to {QL_{2}(\Omega)}$ is invertible, and the inverse operator coincides
with $R_{1}(0)$.
Substituting representation \eqref{4.10} into the integrals in  \eqref{G0=}--\eqref{Gkl=} and considering
the fact that the operator-function $R_{1}(\zeta)=R_{0}(\zeta)(I-P)$ is holomorphic inside the conrour $\Gamma$,
we conclude that
\begin{equation}
\label{G0=0}
G_{0}= - \frac{1}{2\pi i}\oint_{\Gamma} \Bigl(-\frac{1}{\zeta}P+R_{1}(\zeta)\Bigr)
\zeta \,d\zeta= 0;
\end{equation}
\begin{multline}
\label{Gj=0}
G_{j}=  \frac{1}{2\pi i}\oint_{\Gamma} \Bigl(-\frac{1}{\zeta}P+R_{1}(\zeta)\Bigr)
\partial_{j}\mathbb{A}(\mathbf{0})
\Bigl(-\frac{1}{\zeta}P+R_{1}(\zeta)\Bigr) \zeta \,d\zeta
\\
=
\frac{1}{2\pi i}\oint_{\Gamma} \frac{1}{\zeta}  P \partial_{j}\mathbb{A}(\mathbf{0}) P  \,d\zeta
=  P \partial_{j}\mathbb{A}(\mathbf{0}) P,\ \ j=1,\dots,d;
\end{multline}
\begin{multline}
\label{Gkl==}
G_{kl}= \frac{1}{2\pi i}\oint_{\Gamma} \Bigl(-\frac{1}{\zeta}P+R_{1}(\zeta)\Bigr)
\partial_{k} \partial_l \mathbb{A}(\mathbf{0})
\Bigl(-\frac{1}{\zeta}P+R_{1}(\zeta)\Bigr) \zeta \,d\zeta
\\
- \frac{1}{2\pi i}\oint_{\Gamma}  \Bigl(-\frac{1}{\zeta}P+R_{1}(\zeta)\Bigr)\partial_{k} \A(\mathbf{0})
\Bigl(-\frac{1}{\zeta}P+R_{1}(\zeta)\Bigr)  \partial_{l} \A(\mathbf{0}) \Bigl(-\frac{1}{\zeta}P+R_{1}(\zeta)\Bigr)
 \zeta  \, d\zeta
 \\
- \frac{1}{2\pi i}\oint_{\Gamma}  \Bigl(-\frac{1}{\zeta}P+R_{1}(\zeta)\Bigr)\partial_{l} \A(\mathbf{0})
\Bigl(-\frac{1}{\zeta}P+R_{1}(\zeta)\Bigr)  \partial_{k} \A(\mathbf{0}) \Bigl(-\frac{1}{\zeta}P+R_{1}(\zeta)\Bigr)
 \zeta  \, d\zeta
 \\
=  P \partial_{k} \partial_{l} \mathbb{A}(\mathbf{0}) P
- P \partial_{k}  \mathbb{A}(\mathbf{0}) R_{1}(0)  \partial_{l}  \mathbb{A}(\mathbf{0}) P
-  P \partial_{l}  \mathbb{A}(\mathbf{0}) R_{1}(0)  \partial_{k}  \mathbb{A}(\mathbf{0}) P
\\
- P \partial_{k}  \mathbb{A}(\mathbf{0})P\partial_{l}  \mathbb{A}(\mathbf{0})R_{1}(0)
- R_{1}(0)\partial_{k}  \mathbb{A}(\mathbf{0})P\partial_{l}  \mathbb{A}(\mathbf{0}) P
\\
- P \partial_{l}  \mathbb{A}(\mathbf{0})P\partial_{k}  \mathbb{A}(\mathbf{0})R_{1}(0)
- R_{1}(0)\partial_{l}  \mathbb{A}(\mathbf{0})P\partial_{k}  \mathbb{A}(\mathbf{0})P,\ \ k,l=1,\dots,d.
\end{multline}
This yields the required representations in \eqref{G_{j}=prop} and \eqref{G_kl=prop}.

Representation \eqref{AF = G2} follows from \eqref{AF=},  \eqref{G0=0}, \eqref{Gj=0} and \eqref{Gkl==}.
\end{proof}

Our next goal is to express the right-hand side of representation \eqref{G_{j}=prop} and the first three terms on the right-hand side in \eqref{G_kl=prop}
in terms of solutions to appropriate auxiliary problems.
Since $P=(\cdot,q_{0})\mathbf{1}_{\Omega}$,  then for an arbitrary $T\in\mathcal{B}(L_{2}(\Omega))$ we have
\begin{equation}\label{PMP}
PTP=(T\mathbf{1}_{\Omega},q_{0})P,\ \ T\in\mathcal{B}(L_{2}(\Omega)).
\end{equation}
Therefore,
\begin{equation}\label{e2.20}
G_{j}=P\partial_{j}\A(\mathbf{0})P=i(w_{j},q_{0})P,\ \
\text{where}\ \ iw_{j} :=\partial_{j}\A(\mathbf{0})\mathbf{1}_{\Omega},\
\ j=1,\dots,d.
\end{equation}
Due to (\ref{e2.1}), the following relations hold:
\begin{multline}
\label{e2.21}
w_{j}(\x)=\overline{w_{j}(\x)}=\intop_{\Omega}\sum_{\n\in\Z^d}(x_{j}-y_{j}+n_{j})a(\x-\y+\n)\mu(\x,\y)\,d\y
\\
= \intop_{\R^d}(x_{j}-y_{j})a(\x-\y)\mu(\x,\y)\,d\y,\ \ \x\in\Omega,\ \ j=1,\dots,d;
\end{multline}
here we took into account the periodicity of the function $\mu$. Combining \eqref{e2.20} and \eqref{e2.21} we conclude that the
operators  $G_{j}$, $j=1,\dots,d$,  take the form
\begin{equation}\label{G_{j}=fin}
G_{j}=i\alpha_{j}P,\ \ \alpha_{j}=(w_{j},q_{0})=\int_{\Omega}d\x\,q_{0}(\x)\int_{\mathbb{R}^{d}}d\y\,(x_{j}-y_{j})a(\x-\y)\mu(\x,\y),\ \ j=1,\dots,d.
\end{equation}
By \eqref{PMP} we obtain the relation
\begin{equation}
\label{e2.25}
P\partial_{k}\partial_{l}\A(\mathbf{0})P=(w_{kl},q_{0})P,\
\ k,l=1,\dots,d,
\end{equation}
where
\begin{equation*}
\label{e2.54a}
w_{kl}=\overline{w_{kl}}=\partial_{k}\partial_{l}\A(\mathbf{0})\mathbf{1}_{\Omega}\in L_{2}(\Omega),
\end{equation*}
or, in  more detailed form,
\begin{multline*}\label{e2.26}
w_{kl}(\x)=\intop_{\Omega}\sum_{\n\in\Z^d}(x_{k}-y_{k}+n_{k})(x_{l}-y_{l}+n_{l})a(\x-\y+\n)\mu(\x,\y)\,d\y
\\
= \intop_{\R^d}(x_{k}-y_{k})(x_{l}-y_{l})a(\x-\y)\mu(\x,\y)\, d\y,\ \
\x\in\Omega,\ \ k,l=1,\dots,d.
\end{multline*}
Finally, considering the periodicity of $q_0$ and $\mu$,  we have
\begin{equation}
\label{w_kl,q0}
\begin{aligned}
(w_{kl},q_0)& =
\intop_{\Omega} d\x\, q_{0}(\x) \intop_{\Omega} d\y\,
\sum_{\mathbf{n} \in \Z^d} (x_{k}-y_{k}+n_{k})(x_{l}-y_{l}+n_{l})a(\x-\y+\mathbf{n})\mu(\x,\y)
\\
&=\intop_{\Omega} d\y\, \intop_{\Omega}\,d\x\, q_{0}(\x)
\sum_{\n\in\Z^d}(x_{k}-y_{k}+n_{k})(x_{l}-y_{l}+n_{l})a(\x-\y+\n)\mu(\x,\y)
\\
&=\intop_{\Omega}d\y\, \intop_{\R^{d}}\,d\x\, q_{0}(\x)(x_{k}-y_{k})(x_{l}-y_{l})a(\x-\y)\mu(\x,\y)
\\
&=\intop_{\Omega}d\x\, \intop_{\R^{d}}\,d\y\, q_{0}(\y)(x_{k}-y_{k})(x_{l}-y_{l})a(\y-\x)\mu(\y,\x),\ \ k,l=1,\dots,d.
\end{aligned}
\end{equation}
Next, taking into account \eqref{e2.20}, we deduce from \eqref{PMP} the relations:
\begin{align}\nonumber
P\partial_{k}\A(\mathbf{0})R_{1}(0)\partial_{l}\A(\mathbf{0})P=(iv_{l},i\widetilde w_{k})P,\ \ \text{where}\ \ iv_{l}:=R_{1}(0)\partial_{l}\A(\mathbf{0})\mathbf{1}_{\Omega}=
iR_{1}(0)w_{l},\\   \label{e2.24}
i\widetilde w_{k}:=\left(\partial_{k}\A(\mathbf{0})\right)^{*}q_{0}, \ \ k,l=1,\dots,d.
\end{align}
The expressions in the definition of $v_{l}$, $l=1,\dots,d$, can be rewritten as follows:
\begin{equation}\label{v_{l}}
v_{l}=\A(\mathbf{0})^{-1}Qw_{l},\ \ Qw_{l}=w_{l}-\alpha_{l}\mathbf{1}_{\Omega},\ \ l=1,\dots,d,
\end{equation}
where the functions $w_{l}\in L_{2}(\Omega)$, $l=1,\dots,d$, and the coefficients $\alpha_{l}$, $l=1,\dots,d$,
are defined in (\ref{e2.21}) and \eqref{G_{j}=fin}, respectively.
Due to \eqref{v_{l}}, the functions $v_{l}=\overline{v_{l}}\in L_{2}(\Omega)$,
$j=1,\dots,d$, are solutions to the cell problems in $\Omega$ that read
\begin{equation*}
\label{e2.22a}
\int\limits_{\Omega}\widetilde
a(\mathbf{0},\x-\y)\mu(\x,\y)(v_{l}(\x)-v_{l}(\y))\,d\y=w_{l}(\x)-\alpha_{l},\ \ \x\in\Omega;\quad
\int\limits_{\Omega}v_{l}(\x)q_{0}(\x)\,d\x=0.
\end{equation*}
Assuming that the functions $v_{l}\in L_{2}(\Omega)$, $l=1,\dots,d$, are extended periodically on the whole
$\R^d$, one can rearrange these problems as follows:
\begin{align}\nonumber
&\intop\limits_{\R^d}
a(\x-\y)\mu(\x,\y)(v_{l}(\x)-v_{l}(\y))\,d\y=\intop\limits_{\R^d}
a(\x-\y)\mu(\x,\y)(x_{l}-y_{l}) \,d\y-\alpha_{l},\ \ \x\in\Omega;
\\     \label{e2.23}
&\intop\limits_{\Omega}v_{l}(\x)q_{0}(\x)\,d\x=0.
\end{align}
In view of \eqref{G_{j}=fin} problems (\ref{e2.23}) are uniquely solvable.

Combining \eqref{e2.1} and \eqref{e2.24} we obtain the relations
\begin{equation}\label{wtw_{k}}
\widetilde w_{k}(\x)=\overline{\widetilde w_{k}(\x)\,}=\intop_{\R^{d}}(x_{k}-y_{k})a(\y-\x)\mu(\y,\x)q_{0}(\y)\,d\y,\ \ \x\in\Omega,\ \ k=1,\dots,d.
\end{equation}

As a consequence of  \eqref{G_kl=prop}, \eqref{G_{j}=fin}--\eqref{e2.24} and \eqref{wtw_{k}} we have
the following statement:

\begin{proposition}
Under the assumptions of Proposition {\rm \ref{prop2.6}} the operators $G_{j}$, $G_{kl}$
admit the representations
\begin{equation*}
G_{j}=i\alpha_{j} P,\ \ j=1,\dots,d,\ \ G_{kl} = g_{kl} P+PG_{kl}Q+QG_{kl}P,
\ \ k,l=1,\dots,d;
\end{equation*}
here $Q=I-P$, the real coefficients  $\alpha_{j}$ are introduced in \eqref{G_{j}=fin},
\begin{multline}
\label{Gkl===}
g_{kl} = (w_{kl},q_0)-(v_{k},\widetilde w_{l})-(v_{l},\widetilde w_{k})
\\
= \intop_{\Omega}d\x\intop_{\R^d} d\y((x_{k}-y_{k})(x_{l}-y_{l})-v_{k}(\x)(x_{l}-y_{l})-v_{l}(\x)(x_{k}-y_{k}))a(\y-\x)\mu(\y,\x)q_{0}(\y),
\end{multline}
and $v_j$ is a periodic solution of  \eqref{e2.23}.
Furthermore,
\begin{equation}\label{G_{j}===}
[G]_{1}(\bxi)=\sum_{j=1}^{d}G_{j}\xi_{j}=i\sum_{j=1}^{d}\alpha_{j}\xi_{j}P=i\langle{\boldsymbol\alpha},\bxi\rangle P,
\quad \bxi   \in \R^d,
\end{equation}
\begin{multline}
\label{Gkl====}
 [G]_2(\bxi) = \frac{1}{2}\sum_{k,l=1}^d G_{kl} \xi_k \xi_l =
 \frac{1}{2}\sum_{k,l=1}^d g_{kl} \xi_k \xi_l P +\frac{1}{2}P\sum_{k,l=1}^{d}G_{kl}\xi_{k}\xi_{l}Q+\frac{1}{2}Q\sum_{k,l=1}^{d}G_{kl}\xi_{k}\xi_{l}P\\
= \langle g^0 \bxi, \bxi\rangle P+\frac{1}{2}P\sum_{k,l=1}^{d}G_{kl}\xi_{k}\xi_{l}Q+\frac{1}{2}Q\sum_{k,l=1}^{d}G_{kl}\xi_{k}\xi_{l}P,\quad \bxi   \in \R^d,
 \end{multline}
  where ${\boldsymbol\alpha}=(\alpha_{1},\dots,\alpha_{d})^{t}$, and $g^0$ is a $(d \times d)$-matrix with the entries $\frac{1}{2} g_{kl}$, $k,l=1, \dots,d$.
\end{proposition}

The matrix $g^0$ is called the \emph{effective matrix};
in view of  \eqref{Gkl===} we infer that $g^{0}$ is real and symmetric; in Section \ref{sec3.1} below
it is shown that the matrix $g^0$ is positive definite.

\section{Approximation of the resolvent $(\A + \eps^2 I)^{-1}$}\label{Sec3}

\subsection{Approximation of the resolvent of operator  $\A(\bxi)$}\label{sec3.1}
In this section we construct an approximation of the resolvent   $(\A(\bxi) + \eps^2 I)^{-1}$ for small $\eps >0$,
the accuracy of this approximation being $O(\eps^{-1})$, see Theorem  \ref{teor2.2} below.
Substituting $F(\bxi)u$ for $u$ in (\ref{e1.30}) yields
\begin{equation}\label{e2.35}
\operatorname{Re}([q_{0}]\A(\bxi;a,\mu)F(\bxi)F(\bxi)u,F(\bxi)u)\ge\mu_{-}q_{-}C(a)|\bxi|^{2}(F(\bxi)u,F(\bxi)u),\ \
u\in L_{2}(\Omega),\ \ |\bxi|\le\delta_{0}(a,\mu).
\end{equation}
Combining the last inequality with expansion \eqref{AF = G2} and considering  \eqref{G_{j}===} and \eqref{Gkl====},
we obtain
\begin{multline}\label{AFFuFu}
\operatorname{Re}\bigl([q_{0}](i\langle{\boldsymbol\alpha},\bxi\rangle P+\langle g^0 \bxi, \bxi\rangle P+P[G]_{2}(\bxi) Q+Q[G]_{2}(\bxi) P+ \Psi(\bxi))F(\bxi)u,F(\bxi)u \bigr)
\\
\ge\mu_{-}q_{-}C(a)|\bxi|^{2}(F(\bxi)u,F(\bxi)u),\ \
u\in L_{2}(\Omega),\ \ |\bxi|\le\delta_{0}(a,\mu).
\end{multline}
Let us rearrange the left-hand side of this inequality. Since the operator $[q_{0}]P=(\cdot,q_{0})q_{0}$ is
self-adjoint, the relation
\begin{equation}\label{3.2a}
\operatorname{Re}([q_{0}]i\langle{\boldsymbol\alpha},\bxi\rangle P F(\bxi)u,F(\bxi)u)=0
\end{equation}
holds true.  Taking into account the relations $QP =0$ and $P^{*}[q_{0}]=[q_{0}]P$, one has
\begin{equation}\label{qPGQFuFu}
\begin{aligned}
&\operatorname{Re}\bigl([q_{0}](P[G]_{2}(\bxi) Q+Q[G]_{2}(\bxi) P+ \Psi(\bxi))F(\bxi)u,F(\bxi)u \bigr)
\\
&=
\operatorname{Re}\bigl([q_{0}]P[G]_{2}(\bxi) Q (F(\bxi) -P)u,F(\bxi)u \bigr)
\\
&+ \operatorname{Re}\bigl([q_{0}]Q[G]_{2}(\bxi) P F(\bxi)u,(F(\bxi) - P)u \bigr)
+ \operatorname{Re}\bigl([q_{0}] \Psi(\bxi)F(\bxi)u,F(\bxi)u \bigr).
\end{aligned}
\end{equation}
Since $q_0 \in L_\infty$, making use of  \eqref{F-P}, \eqref{AF = G2} and \eqref{AF-G} we conclude that
the absolute value of the right-hand side in \eqref{qPGQFuFu} in not greater than $C' |\bxi|^3 \| u \|^2$
with some constant $C'$.
In view of \eqref{AFFuFu} and \eqref{3.2a} this leads to the inequality
\begin{multline}\label{AFFuFu==}
\operatorname{Re}([q_{0}]\langle g^0 \bxi, \bxi\rangle PF(\bxi)u,F(\bxi)u)
\\
\ge\mu_{-}q_{-}C(a)|\bxi|^{2}(F(\bxi)u,F(\bxi)u) - C' |\bxi|^3 \|u\|^2,\ \
u\in L_{2}(\Omega),\ \ |\bxi|\le\delta_{0}(a,\mu).
\end{multline}
Using  \eqref{F-P} one more time, we have
\begin{multline*}
\operatorname{Re}([q_{0}]\langle g^0 \bxi, \bxi\rangle Pu,Pu)
\ge\mu_{-}q_{-}C(a)|\bxi|^{2}(Pu,Pu) - C'' |\bxi|^3 \|u\|^2,\ \
u\in L_{2}(\Omega),\ \ |\bxi|\le\delta_{0}(a,\mu),
\end{multline*}
with some constant $C''$.
Substituting here $\mathbf{1}_{\Omega}$ for $u$ yields
\begin{equation*}
\langle g^0 \bxi,\bxi \rangle \ge\mu_{-}q_{-}C(a)|\bxi|^{2} - C'' |\bxi|^{3},
\quad |\bxi|\le\delta_{0}(a,\mu).
\end{equation*}
It remains to divide the last inequality by $|\bxi|^{2}$ and send $|\bxi|$ to zero.
The resulting lower bound reads
\begin{equation*}
 \langle g^0 \boldsymbol{\theta},\boldsymbol{\theta} \rangle \ge\mu_{-}q_{-}C(a),\
\ \boldsymbol{\theta}\in\mathbb{S}^{d-1},
\end{equation*}
or, equivalently,
\begin{equation}\label{e2.36}
\langle g^0 \bxi,\bxi \rangle \ge\mu_{-}q_{-}C(a)|\bxi|^{2},\
\ \bxi\in\R^d.
\end{equation}
This yields the desired positive definiteness of $g^0$.

In order to formulate the next statement we first introduce the notation
\begin{equation*}
\label{Xi=}
\Xi(\bxi,\eps) := (\A(\bxi)+\varepsilon^{2}I)^{-1} F(\bxi) - \left( \langle g^0 \bxi,\bxi \rangle + i\langle {\boldsymbol\alpha},\bxi\rangle + \eps^2 \right)^{-1} P,
 \ \ \bxi\in\widetilde{\Omega}, \ \ \eps >0.
\end{equation*}
\begin{proposition}
Let conditions \eqref{h1.1}--\eqref{h1.3}  be satisfied, and assume that $M_3(a) < \infty$.
Then the following estimate holds:
\begin{equation}
\label{Xi=le}
\| \Xi(\bxi,\eps) \| \le
\frac{ C_3(a,\mu)  |\bxi|}{\mu_{-}q_{-}q_{+}^{-1}C(a)|\bxi|^{2} + \eps^2}
+
\frac{C_4(a,\mu) |\bxi|^3}{(\mu_{-}q_{-}q_{+}^{-1}C(a)|\bxi|^{2} + \eps^2)^2},
\quad |\bxi|\le\delta_{0}(a,\mu), \quad \eps>0,
\end{equation}
where
\begin{align}
\label{C33}
C_3(a,\mu) &= \frac{5}{4} K d_{0}q_{-}^{-1/2}q_{+}^{1/2} C_1(a,\mu),
\\
\label{C4}
C_4(a,\mu) &= \frac{1}{2}K d_{0}q_{-}^{-1}q_{+}\left(\frac{3}{4}K d_{0}C_2(a,\mu)+\frac{1}{2}C_{1}(a,\mu)S(d_{0},K,a,\mu)d \left(1 + \frac{1}{2} K d_0 \right)\right),
\end{align}
and the constant $S(d_{0},K,a,\mu)$ is defined in  \eqref{GklG2}.
\end{proposition}

\begin{proof}
It follows from  (\ref{e1.30.1}) and (\ref{e2.36}) that
\begin{align}\label{e2.39}
&\|(\A(\bxi)+\varepsilon^{2}I)^{-1} \|
 \le q_{-}^{-1/2}q_{+}^{1/2}(\mu_{-}q_{-}q_{+}^{-1}C(a)|\bxi|^{2}+\varepsilon^{2})^{-1}:=\mathcal{S}(\bxi,\varepsilon),\
\ \varepsilon>0,\ \ \bxi \in \wt{\Omega};
\\
\label{e2.40}
&\left|(\langle g^0 \bxi,\bxi \rangle + i\langle {\boldsymbol\alpha},\bxi\rangle +\varepsilon^{2})^{-1}\right|
 \le(\mu_{-}q_{-}C(a)|\bxi|^{2}+\varepsilon^{2})^{-1}\le \mathcal{S}(\bxi,\varepsilon),\
\ \varepsilon>0,\ \ \bxi \in \wt{\Omega};
\end{align}
in the second inequality in \eqref{e2.40} the lower bound $q_{+}\ge 1$ has been used.
In view of \eqref{e2.6} and \eqref{e2.32},
\begin{equation}\label{FPQ}
\|F(\bxi)\|\le\frac{3}{4}K d_{0},\ \ |\bxi|\le\delta_{0}(a,\mu);\ \ \|P\|\le\frac{1}{2}
 K d_{0};\ \ \|Q\|\le 1+\frac{1}{2} K d_{0}.
\end{equation}
Next, representation  \eqref{R1} yields the inequality
\begin{equation}\label{|R1|}
\|R_{1}(0)\|\le K.
\end{equation}
Due to \eqref{2.4} and \eqref{e2.3} the first order derivatives $\partial_{j}\A(\boldsymbol{0})$ admit the upper bounds
\begin{equation}\label{d_{j}A}
\|\partial_{j}\A(\boldsymbol{0})\|\le\mu_{+}M_{1}(a),\ \ j=1,\dots,d,
\end{equation}
while, in view of \eqref{2.6} and \eqref{e2.4}, the second order derivatives  $\partial_{k}\partial_{l}\A(\boldsymbol{0})$
satisfy the estimates
\begin{equation}\label{d_{k}d_{l}A}
\|\partial_{k}\partial_{l}\A(\boldsymbol{0})\|\le \mu_{+}M_{2}(a),\ \ k,l=1,\dots,d.
\end{equation}
Representation  \eqref{Gkl==} and estimates  \eqref{FPQ}--\eqref{d_{k}d_{l}A} lead to the relations
\begin{align}\label{GklG2}
\|G_{kl}\|&\le S(d_{0},K,a,\mu):=\left(\mu_{+}M_{2}(a)+
6(\mu_{+}M_{1}(a))^{2}K \right)\left(\frac{1}{2} K d_{0}\right)^{2},\ \ k,l=1,\dots,d;\\
\label{GklG22}
\|[G]_{2}(\bxi)\|&\le \frac{d}{2} S(d_{0},K,a,\mu)|\bxi|^{2},\ \ \bxi\in\widetilde\Omega.
\end{align}
Writing down an evident identity
\begin{multline}\label{e2.41}
\Xi(\bxi,\eps) =
F(\bxi)(\A(\bxi)+\varepsilon^{2}I)^{-1}(F(\bxi)-P)+(F(\bxi)-P) (\langle g^0 \bxi,\bxi \rangle + i\langle {\boldsymbol\alpha},\bxi\rangle +\varepsilon^{2})^{-1} P
-
\\
-
F(\bxi)(\A(\bxi)+\varepsilon^{2}I)^{-1}
\left(\A(\bxi)F(\bxi) - (\langle g^0 \bxi,\bxi \rangle + i\langle {\boldsymbol\alpha},\bxi\rangle)P \right)
(\langle g^0 \bxi,\bxi \rangle + i\langle {\boldsymbol\alpha},\bxi\rangle +\varepsilon^{2})^{-1} P,
\end{multline}
one can rearrange the third term on its right-hand side as follows:
\begin{multline}\label{tretie}
-
F(\bxi)(\A(\bxi)+\varepsilon^{2}I)^{-1}
\left(\A(\bxi)F(\bxi) - (\langle g^0 \bxi,\bxi \rangle + i\langle {\boldsymbol\alpha},\bxi\rangle)P \right)
(\langle g^0 \bxi,\bxi \rangle + i\langle {\boldsymbol\alpha},\bxi\rangle +\varepsilon^{2})^{-1} P\\
=-
F(\bxi)(\A(\bxi)+\varepsilon^{2}I)^{-1}\Psi(\bxi)(\langle g^0 \bxi,\bxi \rangle + i\langle {\boldsymbol\alpha},\bxi\rangle +\varepsilon^{2})^{-1} P-\\
-
(\A(\bxi)+\varepsilon^{2}I)^{-1}(F(\bxi)-P)Q[G]_{2}(\bxi)(\langle g^0 \bxi,\bxi \rangle + i\langle {\boldsymbol\alpha},\bxi\rangle +\varepsilon^{2})^{-1} P;
\end{multline}
here \eqref{AF = G2}, \eqref{G_{j}===} and \eqref{Gkl====} have been used.
Now the required estimate \eqref{Xi=le}  is a consequence of  \eqref{F-P}, \eqref{AF-G} and
(\ref{e2.39})--(\ref{FPQ}), \eqref{GklG22}--\eqref{tretie}.
\end{proof}

\begin{theorem}\label{teor2.2}
Let conditions  \eqref{h1.1}--\eqref{h1.3} be fulfilled, and assume that $M_3(a) < \infty$.
Then the following estimate holds:
\begin{equation}\label{e2.37}
\left\|(\A(\bxi)+\varepsilon^{2}I)^{-1}-
(\langle g^0 \bxi,\bxi \rangle +i\langle {\boldsymbol\alpha},\bxi\rangle+\varepsilon^{2})^{-1} P \right\|\le
C_5(a,\mu)\varepsilon^{-1},\ \ \varepsilon>0,\ \
|\bxi|\le\delta_{0}(a,\mu),
\end{equation}
where the constant $C_5(a,\mu)$ depends only on $d_{0}$, $K$, $q_{-}$, $q_{+}$, $d$, $\mu_-$, $\mu_+$, $M_1(a)$, $M_2(a)$, $M_3(a)$, ${\mathcal M}(a)$, ${\mathcal C}_\pi(a)$ and
${\mathcal C}_{r(a)}(a)$.
\end{theorem}

\begin{proof}
From  \eqref{e2.32} and an obvious representation
\begin{equation*}
R(\bxi,-\varepsilon^{2})(I-F(\bxi))=\frac{1}{2\pi i}\oint_{\Gamma}(\zeta+\varepsilon^{2})^{-1}R(\bxi,\zeta)\,d\zeta,\ \
0<\varepsilon^{2}<d_{0}/2,\ \ |\bxi|\le\delta_{0}(a,\mu),
\end{equation*}
it follows that
\begin{equation}\label{e2.38}
\|(\A(\bxi)+\varepsilon^{2}I)^{-1}(I-F(\bxi))\|\le
3K,\ \ 0<\eps^{2}\le d_{0}/4,\ \ |\bxi|\le\delta_{0}(a,\mu).
\end{equation}
On the other hand, relations \eqref{e1.30.1} and \eqref{FPQ} imply the estimate
\begin{equation}\label{e2.38.1}
\|(\A(\bxi)+\varepsilon^{2}I)^{-1}(I-F(\bxi))\|\le
q_{-}^{-1/2}q_{+}^{1/2}\Bigl(1+\frac{3}{4}K d_{0}\Bigr)\eps^{-2},\ \ \eps>0,\ \ |\bxi|\le\delta_{0}(a,\mu).
\end{equation}
By the last two inequalities we have
\begin{equation}\label{e2.38a}
\|(\A(\bxi)+\varepsilon^{2}I)^{-1}(I-F(\bxi))\|\le \sqrt{3K}
q_{-}^{-1/4}q_{+}^{1/4} \Bigl(1+\frac{3}{4}K d_{0}\Bigr)^{1/2} \eps^{-1},\ \ 0 < \eps \le \frac{\sqrt{d_0}}{2},\ \ |\bxi|\le\delta_{0}(a,\mu).
\end{equation}
Inequality  \eqref{e2.38.1} also implies that
\begin{equation}
\label{e2.38.2}
\|(\A(\bxi)+\varepsilon^{2}I)^{-1}(I-F(\bxi))\|\le
q_{-}^{-1/2}q_{+}^{1/2}\Bigl(1+\frac{3}{4}K d_{0}\Bigr)2d_0^{-1/2} \eps^{-1},\ \  \eps > \frac{\sqrt{d_0}}{2},\ \ |\bxi|\le\delta_{0}(a,\mu).
\end{equation}
Combining \eqref{e2.38a} and \eqref{e2.38.2} yields
\begin{equation}
\label{e2.38.3}
\begin{aligned}
\|(\A(\bxi)+\varepsilon^{2}I)^{-1}(I-F(\bxi))\|\le C_5^{(1)} \eps^{-1},\ \  \eps > 0,\ \ |\bxi|\le\delta_{0}(a,\mu),
\\
C_5^{(1)} = \max \left\{ \sqrt{3K}
q_{-}^{-1/4}q_{+}^{1/4} \Bigl(1+\frac{3}{4}K d_{0}\Bigr)^{1/2}, q_{-}^{-1/2}q_{+}^{1/2}\Bigl(1+\frac{3}{4}K d_{0}\Bigr)2d_0^{-1/2} \right\}.
\end{aligned}
\end{equation}

Directly from \eqref{Xi=le}  we obtain
\begin{equation}
\label{Xi=le_2}
\begin{aligned}
\| \Xi(\bxi,\eps) \| \le C_5^{(2)} \eps^{-1},\quad |\bxi|\le\delta_{0}(a,\mu), \quad \eps>0,
 \\
 C_5^{(2)} =  \frac{ C_3(a,\mu)}{(\mu_{-}q_{-}q_{+}^{-1}C(a))^{1/2}}
+\frac{C_4(a,\mu)}{(\mu_{-}q_{-}q_{+}^{-1}C(a))^{3/2}}.
\end{aligned}
\end{equation}
Clearly, the operator under the norm sign on the left-hand side of inequality \eqref{e2.37} is equal to
$(\A(\bxi)+\varepsilon^{2}I)^{-1}(I-F(\bxi)) +\Xi(\bxi,\eps)$,
therefore,  the desired estimate \eqref{e2.37} follows from \eqref{e2.38.3} and \eqref{Xi=le_2},
the constant $C_5(a,\mu)$ being equal to $C_5^{(1)}(a,\mu) + C_5^{(2)}(a,\mu)$.
Considering the relations in \eqref{C(a)}, \eqref{C1}, \eqref{C3}, \eqref{C33}, \eqref{C4} and \eqref{GklG2} we infer
that $C_5(a,\mu)$ depends only on the parameters $d_{0}$, $K$, $q_{-}$, $q_{+}$, $d$, $\mu_-$, $\mu_+$, $M_1(a)$, $M_2(a)$, $M_3(a)$, ${\mathcal M}(a)$, ${\mathcal C}_\pi(a)$ and
${\mathcal C}_{r(a)}(a)$.
\end{proof}

We introduce the {\it effective operator} as a self-adjoint second order elliptic differential operator
with constant coefficients in $L_2(\R^d)$ that has the form
\begin{equation}
\label{eff_op}
\A^{0}:=\frac{1}{2}\sum_{k, l=1}^{d}g_{kl}D_{k}D_{l} = - \operatorname{div} g^0 \nabla,
\quad \operatorname{Dom} \A^0 = H^2(\R^d);
\end{equation}
here the \emph{effective matrix}  $g^0$ is a $(d\times d)$-matrix with the entries  $\frac{1}{2}g_{kl}$, $k,l=1,\dots,d$,
defined in \eqref{Gkl===}. Due to  (\ref{e2.36}) the matrix  $g^0$ is positive definite.

With the help of the unitary Gelfand transform the operator $\A^0$ can be decomposed into a direct integral:
\begin{equation}
\label{direct_int}
\A^0 = {\mathcal G}^*\Bigl( \int_{\wt{\Omega}} \oplus \A^0(\bxi)\,d\bxi  \Bigr) {\mathcal G};
\end{equation}
here  $\A^0(\bxi)$ is a  self-adjoint operator in $L_2(\Omega)$ defined by the expression
$$
\A^0(\bxi) = (\D + \bxi)^* g^0 (\D+\bxi),
\quad \operatorname{Dom} \A^0 (\bxi) = \wt{H}^2(\Omega).
$$
The space $\wt{H}^2(\Omega)$ consists of  $H^2(\Omega)$ functions whose $\Z^d$-periodic extension in $\R^d$
belongs to  $H^2_{\operatorname{loc}}(\R^d)$.
The relation in \eqref{direct_int} is interpreted as follows: if $u \in \operatorname{Dom} \A^0 = H^2(\R^d)$
and $v = \A^0 u$, then
 $\mathcal{G}{u}(\bxi,\cdot) \in \operatorname{Dom} \A^0(\bxi) = \wt{H}^2(\Omega)$ and
$\mathcal{G}{v}(\bxi,\cdot) = \A^0(\bxi) \mathcal{G}{u}(\bxi,\cdot)$, $\bxi \in \wt{\Omega}$.

Next we define the operator $\A^{0}+i\langle\D,\boldsymbol{\alpha}\rangle$ with the domain
$\operatorname{Dom} \A^0$. This operator can also be represented as a direct integral
\begin{equation}
\label{direct_int1}
\A^0+i\langle\D,\boldsymbol{\alpha}\rangle = {\mathcal G}^*\Bigl( \int_{\wt{\Omega}} \oplus (\A^0(\bxi)+i\langle\D+\bxi,\boldsymbol{\alpha}\rangle)\,d\bxi  \Bigr) {\mathcal G}.
\end{equation}
The operator $\A^0(\bxi)+i\langle\D+\bxi,\boldsymbol{\alpha}\rangle$  is equipped with the domain
$\operatorname{Dom} \A^0 (\bxi)$.

 As a consequence of Theorem \ref{teor2.2} we have

\begin{theorem}\label{teor3.3}
Let conditions \eqref{h1.1}--\eqref{h1.3} be satisfied, and assume that  $M_3(a) < \infty$.
Then the following upper bound is valid:
\begin{equation}\label{e2.37_2}
\left\|(\A(\bxi)+\varepsilon^{2}I)^{-1}- (\A^0(\bxi)+i\langle\D+\bxi,\boldsymbol{\alpha}\rangle+\varepsilon^{2}I)^{-1}[q_{0}]
\right\|_{L_2(\Omega) \to L_2(\Omega)}\le
{\mathrm C}_1(a,\mu)\varepsilon^{-1},\ \ \varepsilon>0,\ \
\bxi \in \wt{\Omega};
\end{equation}
here the constant  ${\mathrm C}_1(a,\mu)$ depends only on the parameters $d_{0}$, $K$, $q_{-}$, $q_{+}$, $d$, $\mu_-$, $\mu_+$, $M_1(a)$, $M_2(a)$, $M_3(a)$, ${\mathcal M}(a)$, ${\mathcal C}_\pi(a)$ and
${\mathcal C}_{r(a)}(a)$.
\end{theorem}

\begin{proof}
For $\bxi \in \wt{\Omega}$ such that  $|\bxi| \ge \delta_0(a,\mu)$ we deduce from  \eqref{e2.39} and \eqref{e2.40}
the inequalities
$$ 
\left\|(\A(\bxi)+\varepsilon^{2}I)^{-1} \right\| \le q_{-}^{-1/2}q_{+}^{1/2}(\mu_- q_{-}q_{+}^{-1}C(a))^{-1/2} (\delta_0(a,\mu))^{-1} \eps^{-1},\quad 
\eps>0, 
$$ 
$$
\left| \left( \langle g^0 \bxi, \bxi \rangle +i\langle {\boldsymbol\alpha},\bxi\rangle+ \eps^2 \right)^{-1}\right| \le q_{-}^{-1/2}q_{+}^{1/2}(\mu_- q_{-}q_{+}^{-1}C(a))^{-1/2} (\delta_0(a,\mu))^{-1} \eps^{-1},\quad 
\eps >0.
$$ 
Combining these inequalities with that of Theorem  \ref{teor2.2} and considering \eqref{FPQ} yields
\begin{equation}\label{e2.37_all}
\left\|(\A(\bxi)+\varepsilon^{2}I)^{-1}-
(\langle g^0 \bxi,\bxi \rangle +i\langle {\boldsymbol\alpha},\bxi\rangle+\varepsilon^{2})^{-1} P \right\|\le
\wt{C}_5(a,\mu)\varepsilon^{-1},\ \ \varepsilon>0,\ \ \bxi \in \wt{\Omega},
\end{equation}
where $\wt{C}_5(a,\mu) = \max \{ C_5(a,\mu),  q_{-}^{-1/2}q_{+}^{1/2}(\mu_- q_{-}q_{+}^{-1}C(a))^{-1/2} (\delta_0(a,\mu))^{-1} (1+Kd_0/2)\}$.

Next, we define the orthogonal projector $P_{0}=(\cdot,\mathbf{1}_{\Omega})\mathbf{1}_{\Omega}$ and
notice that $P=P_{0}[q_{0}]$,
and that the following obvious relation holds:
\begin{equation*}
\label{A0_P}
(\A^0(\bxi)+i\langle\D+\bxi,\boldsymbol{\alpha}\rangle) P_{0} = (\langle g^0 \bxi,\bxi \rangle+i\langle\boldsymbol{\alpha},\bxi\rangle) P_{0}.
\end{equation*}
Therefore,
\begin{equation*}
\label{res0_P}
(\A^0(\bxi) +i\langle\D+\bxi,\boldsymbol{\alpha}\rangle+ \eps^2 I)^{-1}P_{0} = \left( \langle g^0 \bxi,\bxi \rangle +
i\langle\boldsymbol{\alpha},\bxi\rangle+ \eps^2 \right)^{-1} P_{0},
\end{equation*}
and  \eqref{e2.37_all} takes the form
\begin{equation}
\label{e2.37_all_other}
\left\| (\A(\bxi)+\varepsilon^{2}I)^{-1}- (\A^0(\bxi)+i\langle\D+\bxi,\boldsymbol{\alpha}\rangle+\varepsilon^{2}I)^{-1}P_{0}[q_{0}]
 \right\|\le
\wt{C}_5(a,\mu)\varepsilon^{-1},\ \ \varepsilon>0,\ \ \bxi \in \wt{\Omega}.
\end{equation}

With the help of the discrete Fourier transform we infer that
\begin{multline*}
\label{discrete}
\left\| (\A^0(\bxi)+i\langle\D+\bxi,\boldsymbol{\alpha}\rangle+\varepsilon^{2}I)^{-1}(I-P_{0}) \right\|\\
= \sup_{0 \ne \n \in \Z^d}
\bigl| \langle g^0 (2\pi \n+\bxi), 2\pi \n +\bxi \rangle+i\langle\boldsymbol{\alpha},2\pi \n +\bxi\rangle + \eps^2
\bigr|^{-1} \le (\mu_- q_{-}C(a) \pi^2 + \eps^2)^{-1};
\end{multline*}
here we took into account \eqref{e2.36} and the fact that $|2\pi \n+\bxi| \ge \pi$ for  $\bxi \in \wt{\Omega}$
and $0 \ne \n \in \Z^d$.
Finally, we arrive at the estimate
$$
\left\| (\A^0(\bxi)+i\langle\D+\bxi,\boldsymbol{\alpha}\rangle+\varepsilon^{2}I)^{-1}(I-P_{0}) \right\| \le (\mu_- q_{-}C(a))^{-1/2} \pi^{-1} \eps^{-1},
\ \ \varepsilon>0,\ \ \bxi \in \wt{\Omega},
$$
which together with \eqref{e2.37_all_other} gives the desired estimate  \eqref{e2.37_2} with the constant
$\mathrm{C}_1(a,\mu) = \wt{C}_5(a,\mu) + q_{+}(\mu_- q_{-}C(a))^{-1/2} \pi^{-1}$.
\end{proof}

\subsection{Approximation of the resolvent  $(\A+\varepsilon^{2}I)^{-1}$}

\begin{theorem}\label{teor3.6}
Let conditions \eqref{h1.1}--\eqref{h1.3} be fulfilled, and assume that  $M_3(a) < \infty$.
Then the following upper bound is valid:
\begin{equation}\label{th3.6_1}
\left\|(\A+\varepsilon^{2}I)^{-1}- (\A^0+i\langle\D,\boldsymbol{\alpha}\rangle+\varepsilon^{2}I)^{-1}[q_{0}]
\right\|_{L_2(\R^d) \to L_2(\R^d)}\leqslant
{\mathrm C}_1(a,\mu)\varepsilon^{-1},\ \ \varepsilon>0.
\end{equation}
\end{theorem}

\begin{proof}
Combining representations \eqref{h1.5} and \eqref{direct_int1} we conclude that the operator
$(\A+\varepsilon^{2}I)^{-1}- (\A^0+i\langle\D,\boldsymbol{\alpha}\rangle+\varepsilon^{2}I)^{-1}[q_{0}]$
can be decomposed, with the help of the Gelfand transform, into a direct integral over the operators
$(\A(\bxi) +\varepsilon^{2}I)^{-1}- (\A^0(\bxi)+i\langle\D+\bxi,\boldsymbol{\alpha}\rangle+\varepsilon^{2}I)^{-1}[q_{0}]$.  Therefore,
\begin{multline*}
\left\|(\A+\varepsilon^{2}I)^{-1}- (\A^0+i\langle\D,\boldsymbol{\alpha}\rangle+\varepsilon^{2}I)^{-1}[q_{0}]
\right\|_{L_2(\R^d) \to L_2(\R^d)}
\\
=
\sup_{\bxi \in \wt{\Omega}} \left\|(\A(\bxi)+\varepsilon^{2}I)^{-1}- (\A^0(\bxi)+i\langle\D+\bxi,\boldsymbol{\alpha}
\rangle+\varepsilon^{2}I)^{-1}[q_{0}]
\right\|_{L_2(\Omega) \to L_2(\Omega)},
\end{multline*}
and estimate  \eqref{th3.6_1} follows from \eqref{e2.37_2}.
\end{proof}

\section{Homogenization of nonlocal convolution type operators}

\subsection{Main results}
Here we assume that conditions  (\ref{h1.1})--(\ref{h1.3}) hold and consider
a family of convolution type operators in  $L_{2}(\R^d)$ that read
\begin{equation*}
\A_{\varepsilon}u(\x):=\varepsilon^{-d-2}\intop_{\R^d}a((\x-\y)/\varepsilon)\mu(\x/\varepsilon,\y/\varepsilon)
(u(\x)-u(\y))\,d\y,\
\ \x\in\R^d,\ \ u\in L_{2}(\R^d),\ \ \varepsilon>0.
\end{equation*}
We recall the definition of the effective operator $\A^0$ which is given in \eqref{eff_op},
and that of the effective matrix  $g^0$ with the entries $\frac{1}{2}g_{kl}$, $k,l=1,\dots,d$, introduced in \eqref{Gkl===}.
We also recall that the elements $\alpha_j$, $j=1,\ldots, d$, of the vector  $\boldsymbol{\alpha}$
are defined in \eqref{G_{j}=fin}.

Our main result is deduced from Theorem  \ref{teor3.6} by means of the unitary scaling transformation.

\begin{theorem}
\label{teor3.1}
Let conditions \eqref{h1.1}--\eqref{h1.3} be fulfilled, and assume that  $M_3(a) < \infty$.
Then the following estimate holds:
\begin{equation}\label{e3.1}
\|(\A_{\varepsilon}+I)^{-1}-(\A^{0}+\eps^{-1}\langle\boldsymbol{\alpha},\nabla\rangle+I)^{-1}[q_{0}^{\eps}]\|_{L_2(\R^d) \to L_2(\R^d)}\le
{\mathrm C}_1(a,\mu)\varepsilon,\ \ \varepsilon>0;
\end{equation}
here the constant  ${\mathrm C}_1(a,\mu)$ depends only on   $d_{0}$, $K$, $q_{-}$, $q_{+}$, $d$, $\mu_-$, $\mu_+$, $M_1(a)$, $M_2(a)$, $M_3(a)$, ${\mathcal M}(a)$, ${\mathcal C}_\pi(a)$ and
${\mathcal C}_{r(a)}(a)$.
\end{theorem}

\begin{proof}
Let us introduce the scaling transformation
\begin{equation*}
T_{\varepsilon}u(\x):=\varepsilon^{d/2}u(\varepsilon \x),\ \
\x\in\R^d,\ \ u\in L_{2}(\R^d), \ \ \eps >0.
\end{equation*}
Then $\{T_{\varepsilon}\}_{\eps>0}$ is a family of unitary operators in $L_2(\mathbb R^d)$.
It is straightforward to check that
$$
\A_{\varepsilon} = \varepsilon^{-2} T_{\varepsilon}^* \A T_{\varepsilon},\quad \eps >0.
$$
Therefore,
\begin{equation}\label{e3.2}
(\A_{\varepsilon}+I)^{-1}= T_{\varepsilon}^{*}\varepsilon^{2}(\A+\varepsilon^{2}I)^{-1}T_{\varepsilon},
\ \ \varepsilon>0.
\end{equation}
A similar relation holds for the effective operator, it reads
$$
\A^0 = \varepsilon^{-2} T_{\varepsilon}^* \A^0 T_{\varepsilon},\quad \eps >0,
$$
and gives the following representation:
\begin{equation}\label{e3.2_eff}
(\A^0+\eps^{-1}i\langle\D,\boldsymbol{\alpha}\rangle+I)^{-1}[q_{0}^{\eps}]= T_{\varepsilon}^{*}\varepsilon^{2}(\A^0+i\langle\D,\boldsymbol{\alpha}\rangle+\varepsilon^{2}I)^{-1}[q_{0}]T_{\varepsilon},
\ \ \varepsilon>0.
\end{equation}
Since the operator $T_\eps$ is unitary, it immediately follows from  \eqref{e3.2} and \eqref{e3.2_eff} that
\begin{multline*}
\|(\A_{\varepsilon}+I)^{-1}-(\A^{0}+\eps^{-1}i\langle\D,\boldsymbol{\alpha}\rangle+I)^{-1}[q_{0}^{\eps}]\|_{L_2(\R^d) \to L_2(\R^d)} \\=
\eps^2 \|(\A + \varepsilon^2 I)^{-1}-(\A^{0}+i\langle\D,\boldsymbol{\alpha}\rangle+\eps^2 I)^{-1}[q_{0}]\|_{L_2(\R^d) \to L_2(\R^d)}.
\end{multline*}
This relation and the estimate obtained in Theorem \ref{teor3.6} yield the required estimate \eqref{e3.1}.
\end{proof}

\subsection{Concluding remarks    }\label{sec4.2}

$ $\\

1. The statement of Theorem \ref{teor3.1} remains valid, if we replace the periodicity lattice $\Z^d$  with an arbitrary periodic lattice in  $\R^d$. In this case the constants in the above estimates will depend not only on the coefficients   $a$ and $\mu$ but
also on the parameters of the lattice.

2. If in the formulation of Theorem \ref{teor3.1} one replaces the condition  $M_3(a)<\infty$
with the weaker condition $\int_{\R^d} |\x|^k a(\x) \,d\x < \infty$,
where $2 < k < 3$,  then the following estimate can be proved:
\begin{equation}\label{4.4-1}
\|(\A_{\varepsilon}+I)^{-1} - (\A^{0}+\eps^{-1}\langle\boldsymbol{\alpha},\nabla\rangle+I)^{-1}[q_{0}^{\eps}]
 \|_{L_2(\R^d) \to L_2(\R^d)} \le C \eps^{k-2}, \quad \eps >0.
\end{equation}
Furthermore, if in the formulation of this theorem we impose the condition  $M_{2}(a)<\infty$ instead of
$M_3(a)<\infty$, then we still can justify the convergence
\begin{equation}\label{4.4-0}
\|(\A_{\varepsilon}+I)^{-1} - (\A^{0}+\eps^{-1}\langle\boldsymbol{\alpha},\nabla\rangle+I)^{-1}[q_{0}^{\eps}]
 \|_{L_2(\R^d) \to L_2(\R^d)} \to 0, \quad \eps \to+0.
\end{equation}
Indeed, under the assumption $\int_{\R^d} |\x|^k a(\x) \,d\x < \infty$,
where $2 < k < 3$, inequality  \eqref{AF-G} takes the form
\begin{equation*}
\left\| \Psi (\bxi)  \right\| \le C_2(a,\mu) |\bxi|^k, \quad  |\bxi| \le \delta_0(a,\mu).
\end{equation*}
This yields \eqref{4.4-1}.
If in the formulation of Theorem \ref{teor3.1} we replace the condition
$M_3(a)<\infty$ with $M_2(a)<\infty$, then
formula  \eqref{AF = G2} still holds, however, the estimate of the remainder
in   \eqref{AF-G} is not valid any more, instead we have the relation
 $\|\Psi(\bxi)\|=o(|\bxi|^2)$, which leads to \eqref{4.4-0}.

3. If conditions \eqref{h1.1}--\eqref{h1.3} are satisfied and $M_2(a) < \infty$, then, in the topology of strong convergence,
\begin{equation}
\label{4.4}
(\A_{\varepsilon}+I)^{-1} - (\A^{0}+\eps^{-1}\langle\boldsymbol{\alpha},\nabla\rangle+I)^{-1}
 \to 0,  \quad \eps \to 0.
\end{equation}

\begin{proof}
The proof of convergence in \eqref{4.4} relies on the so-called "mean value property'', see, for instance, \cite{JKO}:
as $\eps\to 0$, the operator $[q_0^\eps]$ converges weakly in $L_2(\mathbb R^d)$  to the operator of multiplication by
the average of $q_0(\x)$, that is by $1$.
Let us show that
\begin{equation}
\label{4.5}
\|  (\A^{0}+\eps^{-1}\langle\boldsymbol{\alpha},\nabla\rangle+I)^{-1} [q_0^\eps] u -
 (\A^{0}+\eps^{-1}\langle\boldsymbol{\alpha},\nabla\rangle+I)^{-1} u \|_{L_2(\R^d)}
 \to 0,  \quad \eps \to 0, \quad u \in L_2(\R^d).
\end{equation}
Since
$$
\| (\A^{0}+\eps^{-1}\langle\boldsymbol{\alpha},\nabla\rangle+I)^{-1}  \|_{L_2 \to L_2} \le 1,
\quad \| [q_0^\eps] \|_{L_2 \to L_2} \le q_+,
$$
it is sufficient to prove the convergence in  \eqref{4.5} only for $u \in C_0^\infty(\R^d)$.
For an arbitrary $u \in C_0^\infty(\R^d)$ denote by $B$ the ball that contains the support of $u$. Then,
the following relation holds:
\begin{equation}
\label{4.6}
\begin{aligned}
&(\A^{0}+\eps^{-1}\langle\boldsymbol{\alpha},\nabla\rangle+I)^{-1} [q_0^\eps] u -
 (\A^{0}+\eps^{-1}\langle\boldsymbol{\alpha},\nabla\rangle+I)^{-1} u
 \\
 &\quad =  (\A^{0}+I)(\A^{0}+\eps^{-1}\langle\boldsymbol{\alpha},\nabla\rangle+I)^{-1}  (\A^0 +I)^{-1} [\mathbf{1}_B
 (q_0^\eps -1)] u.
 \end{aligned}
\end{equation}
Since the norm of the operator $(\A^{0}+I)(\A^{0}+\eps^{-1}\langle\boldsymbol{\alpha},\nabla\rangle+I)^{-1}$
is not greater than one, the operator $ (\A^0 +I)^{-1} [\mathbf{1}_B]$ is compact in  $L_2(\R^d)$,
and $(q_0^\eps -1) u \to 0$ weakly in $L_2(\R^d)$, then the $L_2$ norm of the expression on the right-hand side of
\eqref{4.6} tends to zero as $\eps \to 0$.

Combining \eqref{4.4-0} and \eqref{4.5} we obtain  \eqref{4.4}.
\end{proof}

\section{Appendix I. Stability of an isolated eigenvalue}
A proof of the following statement can be found in  \cite[Sec. I.4.6]{K}.

\begin{proposition}\label{rizprp25}
Let  $P_{1}$ and $P_{2}$ be projectors in $\H$, and assume that $\|P_{1}-P_{2}\|<1$.
Then  $\operatorname{rank}P_{1}= \operatorname{rank}P_{2}$.
\end{proposition}

Given a bounded linear operator $A$ in a Hilbert space $\H$, assume that $\lambda_{0}$ is a simple eigenvalue of $A$,
so that the Riesz projector of $A$ that corresponds to the point $\lambda_0$ has rank $1$.
Then, denoting by $d_{0}$ the distance from $\lambda_{0}$ to the remaining part of the spectrum of $A$, it is straightforward to observe that
the resolvent $(A-\zeta I)^{-1}$ is holomorphic in the punctured disc $B_{d_{0}}(\lambda_{0})\setminus \{\lambda_{0}\}$,
and thus the following quantity is finite:
$$
K :=\max\limits_{\frac{d_{0}}{3}\le|\zeta-\lambda_{0}|\le\frac{2d_{0}}{3}}\|(A-\zeta I)^{-1}\| <\infty.
$$
\begin{proposition}\label{rizprp3}
Let $B$ be a bounded linear operator in $\H$, and assume that the norm $\|A-B\|$  is so small that
the following estimates hold:
\begin{equation}\label{rizprp3.1}
 K \|A-B\|<1,\ \ \frac{d_{0}}{3}\cdot\frac{K^{2}\|A-B\|}{1- K\|A-B\|}<1.
\end{equation}
Then

{\rm 1)} the annulus $\{\zeta\in\C:\frac{d_{0}}{3}\le|\zeta-\lambda_{0}|\le\frac{2d_{0}}{3}\}$ and the spectrum  $\sigma(B)$ do not intersect\textup{;}

{\rm 2)}the spectrum of the operator $B$ inside the disc   $B_{d_0/3}(\lambda_0)$
consists of one simple eigenvalue, that is the Riesz projector corresponding to the spectrum of $B$ in the disc
$B_{d_0/3}(\lambda_0)$ has rank $1$.
\end{proposition}

\begin{remark}\label{rizremark}
Relations in  \eqref{rizprp3.1} are equivalent to the inequality
$
\|A-B\|<3(d_{0}K^{2}+3K)^{-1}.
$
\end{remark}

\begin{proof}
Under conditions \eqref{rizprp3.1}, any $\zeta\in\C$ such that $\frac{d_{0}}{3}\le|\zeta-\lambda_{0}|\le\frac{2d_{0}}{3}$
belongs to the resolvent set of the operator $B$; moreover, the resolvent $(B-\zeta I)^{-1}$ admits the representation
$$
(B-\zeta I)^{-1}=\sum_{n=0}^{\infty}((A-\zeta I)^{-1}(A-B))^{n}(A-\zeta I)^{-1}.
$$
This yields the first statement of Proposition.
To prove the second one we use the estimate
\begin{equation}\label{rizprp3.2}
\|(B-\zeta I)^{-1}-(A-\zeta I)^{-1}\|\le \frac{K^{2}\|A-B\|}{1-K\|A-B\|},\ \ |\zeta-\lambda_{0}|=d_{0}/3.
\end{equation}
Denote by $P_{A}$ and $P_{B}$ the Riesz projectors of the operators $A$ and $B$ that correspond
to the parts of their spectra located inside the disc  $B_{d_{0}/3}(\lambda_{0})$.
Combining the formula
$$
P_{A}-P_{B}=\frac{-1}{2\pi i}\oint_{|\zeta-\lambda_{0}|=d_{0}/3}((A-\zeta I)^{-1}-(B-\zeta I)^{-1})d\zeta,
$$
inequality \eqref{rizprp3.2} and condition  \eqref{rizprp3.1},  we conclude that
$
\|P_{A}-P_{B}\|<1;
$
this leads to the second statement of Proposition \ref{rizprp3}, see Proposition \ref{rizprp25}.
\end{proof}

\section{Appendix II. Some properties of operator  $\mathbf{G}$}\label{Sec5}

\subsection{Formulation of the result}
Here we assume that the functions $a(\z)$, $\z\in\mathbb{R}^{d}$, and $\mu(\x,\y)$, $\x,\y\in\mathbb{R}^{d}$, satisfy conditions \eqref{h1.1} and \eqref{h1.2}, \eqref{h1.3}, respectively.

Consider the function
\begin{equation}
\label{5.1}
\wt{a}(\z) := \wt{a}(\mathbf{0},\z) = \sum_{\n \in \Z^d} a(\z +\n), \quad \z \in \R^d,
\end{equation}
see \eqref{a_tilde} for the definition of $\wt{a}(\bxi,\z)$. From \eqref{h1.1} and \eqref{5.1} it follows that
 $\wt{a}(\z)$ is a non-negative  $\Z^d$-periodic  function, $\wt{a} \in L_1(\Omega)$,  and
\begin{equation}
\label{5.2}
 \| \wt{a} \|_{L_1(\Omega)} = \| a \|_{L_1(\R^d)} >0.
 \end{equation}

Recalling the definitions of $\A(\mathbf{\bxi})$ and $\B(\mathbf{\bxi})$ in \eqref{A_xi}, \eqref{B_xi},
we introduce in $L_2(\Omega)$ the following bounded operators: $\mathbf{A}:= \A(\mathbf{0})$ and
$\mathbf{B}:= \B(\mathbf{0})$. Then
\begin{equation*}
\label{A(0)}
\begin{aligned}
\mathbf{A} u (\x) &= p(\x) u(\x) - \mathbf{B} u(\x),
\\
\mathbf{B}u(\x)  &= \intop_{\Omega}\widetilde a(\x-\y)\mu(\x,\y)u(\y)\,d\y,\ \ u\in L_{2}(\Omega),
\end{aligned}
\end{equation*}
with
$$
p(\x) = \intop_\Omega \widetilde a(\x-\y)\mu(\x,\y)\,d\y,
$$
see \eqref{e1.9}. 
According to \eqref{h1.4} the function $p(\x)$ satisfies the estimates
\begin{equation*}
\label{p_est}
  \mu_- \| \wt{a}\|_{L_1(\Omega)} \le p(\x) \le \mu_+ \| \wt{a}\|_{L_1(\Omega)},\quad \x \in \Omega.
\end{equation*}

As was shown in \cite[Corollary 4.2]{PSlSuZh}, the operators $\mathbf {B}$ and $\mathbf{B}^{*}$ are compact,
so is the operator $\mathbf{G}=\mathbf{B}^{*}[p^{-1}]$.  Let us examine the properties of the operator
$\mathbf{G}$ in more detail.
It is clear that $\mathbf{G}^{*}\mathbf{1}_{\Omega}=\mathbf{1}_{\Omega}$, and thus,
$1\in\sigma(\mathbf{G})\cap\sigma(\mathbf{G}^{*})$.
Since the operators $\mathbf{G}$ and $\mathbf{G}^{*}$ are compact, there exists $r>0$ such that
$$
\sigma(\mathbf{G})\cap B_{2r}(1)=\{1\},\ \ \sigma(\mathbf{G}^{*})\cap B_{2r}(1)=\{1\}.
$$
Consequently, the Riesz projectors of the operators $\mathbf{G}$ and $\mathbf{G}^{*}$ that correspond to the point
$1$ can be calculated as follows:
\begin{equation*}
\mathcal{P}=\frac{-1}{2\pi i}\oint_{|\zeta-1|=r}(\mathbf{G}-\zeta I)^{-1}d\,\zeta,\ \
\mathcal{P}^{*}=\frac{-1}{2\pi i}\oint_{|\zeta-1|=r}(\mathbf{G}^{*}-\zeta I)^{-1}d\,\zeta.
\end{equation*}

\begin{theorem}\label{theorem1.5}
The ranks of the  Riesz projectors $\mathcal{P}$ and $\mathcal{P}^{*}$ are equal to  $1$\textup{;}
the kernels of $\mathbf{G}^{*}-I$ and $\mathbf{G}-I$ are defined by
\begin{equation*}\label{f0}
\operatorname{Ker}(\mathbf{G}^{*}-I)=\mathcal{L}\{\mathbf{1}_{\Omega}\},\quad
\operatorname{Ker}(\mathbf{G}-I)=\mathcal{L}\{\psi_{0}\},
\end{equation*}
where the function $\psi_{0}\in L_{2}(\Omega)$ satisfies the relations
\begin{equation*}\label{f1}
0<\psi_{-}\le\psi_{0}(\z)\le\psi_{+}<+\infty\ \ \hbox{\rm for all } \z\in\Omega;\ \ \int_{\Omega}\psi_{0}(\z)\,d\z=1.
\end{equation*}
\end{theorem}

\subsection{}  We begin the proof by obtaining several auxiliary results.
In what follows the notation $(\cdot,\cdot)$ stands for the inner product in  $L_{2}(\Omega)$.

\begin{lemma}\label{lemmaf1.1}
Each non-trivial real function $\psi\in\operatorname{Ker}(\mathbf{G}-I)$ does not change sign and is separated
from zero.
\end{lemma}

\begin{proof}
If $\psi\in L_2(\mathbb R^d)$ is a non-zero real function from $\operatorname{Ker}(\mathbf{G}-I)$, then for any
$N\in\mathbb{N}$  the following identity holds:
\begin{equation}\label{f2}
\mathbf{G}^{N}\psi=\psi.
\end{equation}
Since $\mathbf{G}$ is an integral operator with the kernel
\begin{equation*}
g_{1}(\x,\y):=\wt a(\y-\x) \frac{\mu(\y,\x)}{p(\y)},\ \ \x,\y\in\Omega,
\end{equation*}
then $\mathbf{G}^{N}$  is also an integral operator whose kernel $g_{N}(\x,\y)$ satisfies the estimate
\begin{equation*}
\left(\frac{\mu_{-}}{\mu_{+}\|\wt a\|_{L_{1}(\Omega)}}\right)^{N}F_{N}(\y-\x)\leqslant
g_{N}(\x,\y)\leqslant
\left(\frac{\mu_{+}}{\mu_{-}\|\wt a\|_{L_{1}(\Omega)}}\right)^{N}F_{N}(\y-\x),\ \ \x,\y\in\Omega,
\end{equation*}
where $F_{1}(\z):=\wt a(\z)$ and $F_{N}(\z):=\int_{\Omega}F_{1}(\mathbf{t})F_{N-1}(\z-\mathbf{t})\,d\mathbf{t}$
for $N>1$.
By direct inspection, we see that $F_{N}(\z)$ is a non-negative $\mathbb{Z}^{d}$-periodic function,  and $\|F_{N}\|_{L_{1}(\Omega)}=\|\wt a\|_{L_{1}(\Omega)}^{N}>0$.
It was shown in \cite[Lemma 4.2]{PiaZhi19}  that there exist $\hat N\in\mathbb{N}$ and $\gamma>0$ such that
$F_{\hat N}(\z)\ge\gamma$ for all $\z\in\Omega$.
Then the kernel $g_{\hat N}(\x,\y)$ is positive and separated from zero.

We fix such a representative of the function $\psi(y)$ for which equality \eqref{f2} is satisfied for all $y\in\Omega$.
Letting $\psi^+(y)=\max\{\psi(y),0\}$ and $\psi^-(y)=-\min\{\psi(y),0\}$, one can decompose the function $\psi$ as
$\psi\!=\!\psi^{+}\!-\!\psi^{-}$. Then identity \eqref{f2} takes the form
\begin{equation*}
\mathbf{G}^{N}\psi^{+}-\psi^{+}=\mathbf{G}^{N}\psi^{-}-\psi^{-}.
\end{equation*}
Assume that both $\psi^-$ and $\psi^+$ are non-zero elements of $L_2(\Omega)$.
For any $\x\in\Omega$ such that  $\psi^{+}(\x)>0$ we have  $\psi^{-}(\x)=0$  and
\begin{multline}\label{f3}
\mathbf{G}^{\hat N}\psi^{+}(\x)-\psi^{+}(\x)=\mathbf{G}^{\hat N}\psi^{-}(\x)=\int_{\Omega}g_{\hat N}(\x,\y)\psi^{-}(\y)\,d\y\ge\\\ge
\left(\frac{\mu_{-}}{\mu_{+}\|\wt a\|_{L_{1}(\Omega)}}\right)^{\hat N}\gamma
\int_{\Omega}\psi^{-}(\y)\,d\y=:c_{1}>0.
\end{multline}
If $\x\in\Omega$ is such that  $\psi^{+}(\x)=0$,  then
\begin{multline}\label{f4}
\mathbf{G}^{\hat N}\psi^{+}(\x)-\psi^{+}(\x)= \mathbf{G}^{\hat N}\psi^{+}(\x)= \int_{\Omega}g_{\hat N}(\x,\y)\psi^{+}(\y)\,d\y\geqslant\\
\geqslant
\left(\frac{\mu_{-}}{\mu_{+}\|\wt a\|_{L_{1}(\Omega)}}\right)^{\hat N}\gamma
\int_{\Omega}\psi^{+}(\y)\,d\y=:c_{2}>0.
\end{multline}
From  \eqref{f3} and  \eqref{f4} it follows that
\begin{multline*}
0=\left(\psi^{+},\left(\mathbf{G}^{*}\right)^{\hat N}\mathbf{1}_{\Omega}-\mathbf{1}_{\Omega}\right)= \left(\mathbf{G}^{\hat N}\psi^{+}-\psi^{+},\mathbf{1}_{\Omega}\right)=\\
=\int_{\Omega}(\mathbf{G}^{\hat N}\psi^{+}(\y)-\psi^{+}(\y))\,d\y\ge\min\{c_{1},c_{2}\}>0.
\end{multline*}
Since the last relation is contradictory, we conclude that either $\psi^{+}=0$ or $\psi^{-}=0$.  Therefore, the function
$\psi(\x)$ does not change sign, and the following relations hold:
\begin{equation*}
|\psi(\x)|=\mathbf{G}^{\hat N}|\psi|(\x)=\int_{\Omega}g_{\hat N}(\x,\y)|\psi(\y)|\,d\y\ge
\left(\frac{\mu_{-}}{\mu_{+}\|\wt a\|_{L_{1}(\Omega)}}\right)^{\hat N}\gamma
\int_{\Omega}|\psi(\y)|\,d\y=:c_{3}>0.
\end{equation*}
Consequently, the function  $\psi(\x)$ is separated from zero.
\end{proof}

\begin{lemma}\label{lemmaf1.2}
There exists a real function $\psi_{0}\in L_{2}(\Omega)$ that belongs to the kernel of the operator $(\mathbf{G}-I)$
and satisfies the lower bound
\begin{equation}\label{fdop1}
\psi_{0}(\z)\ge\psi_{-}>0,\ \ \z\in\Omega.
\end{equation}
\end{lemma}
\begin{proof}
Since the integral kernel $g_1(\x,\y)$ of the operator $\mathbf{G}$ is real,  then we can assume that a non-trivial eigenfunction $\psi_{0}\in\operatorname{Ker}(\mathbf{G}-I)$ is real.
By Lemma \ref{lemmaf1.1}  the function $\psi_{0}$ does not change sign and is separated from zero.
Then either $\psi_{0}$ or $-\psi_{0}$ satisfies condition \eqref{fdop1}.
\end{proof}


\begin{lemma}\label{lemmaf1.3}
Let $\psi_{0}$ be an element of $\operatorname{Ker}(\mathbf{G}-I)$ that satisfies condition \eqref{fdop1},
and assume that $\psi_{1}$ is a solution of the equation
$$
(\mathbf{G}-I)\psi_{1}=\psi_{0}.
$$
Then the function $\psi_{1}$ does not change sign.
\end{lemma}

\begin{proof}
Since $\mathbf{G}\psi_0=\psi_0$, we obviously have
\begin{equation}\label{f5}
(\mathbf{G}^{N}-I)\psi_{1}=N\psi_{0}\quad\hbox{for any }n\in\mathbb N.
\end{equation}
Let us choose a representative of the function $\psi_1$ for which identity \eqref{f5} holds pointwise in $\Omega$.
Dividing the function $\psi_{1}$ into its positive and negative parts, $\psi_{1}=\psi_{1}^{+}-\psi_{1}^{-}$,
we assume, by contradiction, that both of these parts are  non-zero. For $N=\hat N$ identity  \eqref{f5} takes the form
\begin{equation*}
\mathbf{G}^{\hat N}\psi_{1}^{+}-\psi_{1}^{+}=\mathbf{G}^{\hat N}\psi_{1}^{-}-\psi_{1}^{-}+\hat N\psi_{0}.
\end{equation*}
If $\x\in\{\y\in\Omega\,:\,\psi_1^+(\y)>0\}$, then $\psi_{1}^{-}(\x)=0$, and
\begin{multline}\label{f6}
\mathbf{G}^{\hat N}\psi_{1}^{+}(\x)-\psi_{1}^{+}(\x)=\mathbf{G}^{\hat N}\psi_{1}^{-}(\x)+\hat N\psi_{0}(\x)= \int_{\Omega}g_{\hat N}(\x,\y)\psi_{1}^{-}(\y)\,d\y+\hat N\psi_{0}(\x)\ge\\\ge
\left(\frac{\mu_{-}}{\mu_{+}\|\wt a\|_{L_{1}(\Omega)}}\right)^{\hat N}\gamma
\int_{\Omega}\psi_{1}^{-}(\y)\,d\y+\hat N\psi_{-}=:s_{1}>0.
\end{multline}
For $\x\in\Omega$ from the complement set we have $\psi_{1}^{+}(\x)=0$ and
\begin{multline}\label{f7}
\mathbf{G}^{\hat N}\psi_{1}^{+}(\x)-\psi_{1}^{+}(\x)=\mathbf{G}^{\hat N}\psi_{1}^{+}(\x)= \int_{\Omega}g_{\hat N}(\x,\y)\psi_{1}^{+}(\y)\,d\y\geqslant \\
\geqslant
\left(\frac{\mu_{-}}{\mu_{+}\|\wt a\|_{L_{1}(\Omega)}}\right)^{\hat N}\gamma
\int_{\Omega}\psi_{1}^{+}(\y)\,d\y=:s_{2}>0.
\end{multline}
Combining  \eqref{f6} and \eqref{f7} yields
\begin{multline*}
0=\left(\psi_{1}^{+},\left(\mathbf{G}^{*}\right)^{\hat N}\mathbf{1}_{\Omega}-\mathbf{1}_{\Omega}\right)= \left(\mathbf{G}^{\hat N}\psi_{1}^{+}-\psi_{1}^{+},\mathbf{1}_{\Omega}\right)= \\
=\int_{\Omega}(\mathbf{G}^{\hat N}\psi_{1}^{+}(\y)-\psi_{1}^{+}(\y))\,d\y\ge\min\{s_{1},s_{2}\}>0.
\end{multline*}
Since the last inequality is contradictory, then either $\psi_{1}^{+}$ or $\psi_{1}^{-}$ is equal to zero. Therefore,
the function $\psi_{1}(\x)$  does not change sign.
\end{proof}

\begin{lemma}\label{lemmaf1.4}
Let a real finction $\psi$ be an element of $\operatorname{Ker}(\mathbf{G}-I)$. Then $\psi=c\psi_0$ with $c\in\mathbb R$.
\end{lemma}

\begin{proof}
Consider the function
$$
u=\psi-\frac{(\psi,\psi_{0})}{(\psi_{0},\psi_{0})}\psi_{0}.
$$
By construction, this function is real and orthogonal to $\psi_{0}$ in $L_2(\Omega)$; moreover,
$u\in\operatorname{Ker}(\mathbf{G}-I)$.
If the function $u$ is not equal to zero, then, by Lemma  \ref{lemmaf1.1}, it does not change sign and thus
$$
|(u,\psi_{0})|=\int_{\Omega}|u(\y)|\psi_{0}(\y)\,d\y\ge\psi_{-}\int_{\Omega}|u(\y)|\,d\y>0.
$$
This contradicts the orthogonality of $u$ and $\psi_0$. Therefore, $u=0$ and
$\psi=\frac{(\psi,\psi_{0})}{(\psi_{0},\psi_{0})}\psi_{0}$.
\end{proof}

It remains to show that the algebraic multiplicity of the eigenvalue $\lambda=1$ of operator $\mathbf{G}$
is equal to one.
\begin{lemma}\label{lemmaf1.5}
The following relation holds:
\begin{equation*}
\operatorname{Ker}(\mathbf{G}-I)^{2}=\operatorname{Ker}(\mathbf{G}-I)=\mathcal{L}\{\psi_{0}\}.
\end{equation*}
\end{lemma}

\begin{proof}
Since $\mathbf{G}$ is an integral operator with a real integral kernel, for any function  $\psi\in\operatorname{Ker}(\mathbf{G}-I)$ one has $\operatorname{Re}\psi\in\operatorname{Ker}(\mathbf{G}-I)$, $\operatorname{Im}\psi\in\operatorname{Ker}(\mathbf{G}-I)$.
By Lemma \ref{lemmaf1.4}, $\operatorname{Re}\psi=\lambda\psi_{0}$, $\operatorname{Im}\psi=\nu\psi_{0}$,
and thus, $\psi=(\lambda+i\nu)\psi_{0}$, i.e. $\operatorname{Ker}(\mathbf{G}-I)=\mathcal{L}\{\psi_{0}\}$.

If we assume that  $\operatorname{Ker}(\mathbf{G}-I)^{2}\setminus\operatorname{Ker}(\mathbf{G}-I)$ is not empty,
then there exists $\psi\in L_2(\Omega)$
 that belongs to
$\operatorname{Ker}(\mathbf{G}-I)^{2}\setminus\operatorname{Ker}(\mathbf{G}-I)$.  In this case
 $(\mathbf{G}-I)\psi=\alpha\psi_{0}$, $\alpha\not=0$ and, therefore, $(\mathbf{G}-I)(\alpha^{-1}\psi)=\psi_{0}$.
Since the integral kernel of $\mathbf{G}$ is real, we obtain
$$
(\mathbf{G}-I)\operatorname{Re}(\alpha^{-1}\psi)=\psi_{0},\ \
(\mathbf{G}-I) \operatorname{Im}(\alpha^{-1}\psi)=0.
$$
Then the function $\psi_1:=\operatorname{Re}(\alpha^{-1}\psi)$ satisfies the relation
\begin{equation*}
(\mathbf{G}-I)\psi_{1}=\psi_{0}.
\end{equation*}
Letting $u=\psi_{1}-\frac{(\psi_{1},\psi_{0})}{(\psi_{0},\psi_{0})}\psi_{0}$, we observe that the function $u$ is real,
orthogonal to the function $\psi_{0}$, and that $(\mathbf{G}-I)u=\psi_{0}$.
By Lemma  \ref{lemmaf1.3} the function $u$ does not change sign. Consequently,
\begin{equation*}
0=|(u,\psi_{0})|=\int_{\Omega}|u(\y)|\psi_{0}(\y)\,d\y\ge\psi_{-}\int_{\Omega}|u(\y)|\,d\y.
\end{equation*}
It follows from the last relation that $u=0$ which contradicts to the identity  $(\mathbf{G}-I) u = \psi_0$. Therefore,
$$
\operatorname{Ker}(\mathbf{G}-I)^{2}=\operatorname{Ker}(\mathbf{G}-I).
$$
\end{proof}

We fix the choice of the function $\psi_{0}\in\operatorname{Ker}(\mathbf{G}-I)$ by imposing the conditions
$$
\psi_{0}(\z)\ge\psi_{-}>0,\ \ \z\in\Omega;\ \ \int_{\Omega}\psi_{0}(\y)\,d\y=1.
$$
The  next statement is a direct consequence of Lemma  \ref{lemmaf1.5}.

\begin{proposition}\label{proposition2}
The Riesz projectors $\mathcal{P}$ and $\mathcal{P}^{*}$ have rank $1$;
the following identities are valid:
$$
\operatorname{Ker}(\mathbf{G}-I)=\mathcal{L}\{\psi_{0}\},\ \ \operatorname{Ker}(\mathbf{G}^{*}-I)=\mathcal{L}\{\mathbf{1}_{\Omega}\};
$$
$$
\mathcal{P}=(\cdot,\mathbf{1}_{\Omega})\psi_{0},\ \ \mathcal{P}^{*}=(\cdot,\psi_{0})\mathbf{1}_{\Omega}.
$$
\end{proposition}

In order to complete the proof of Theorem  \ref{theorem1.5} it remains to show that the function $\psi_{0}$ is bounded.
This is the subject of the following four subsections.
\subsection{Approximation of the kernel $\wt a$ by bounded functions}

For each $N \in \N$ define the function
\begin{equation}
\label{5.3}
  \wt{a}_N(\z) := \begin{cases}\wt{a}(\z), & \text{if}\ \wt{a}(\z) \le N, \\ N, & \text{if}\ \wt{a}(\z) > N.\end{cases}
  \end{equation}
  By construction,  $\wt{a}_N(\z)$ is a $\Z^d$-periodic function such that
 $$
 0 \le \wt{a}_N(\z) \le \wt{a}(\z),\quad \z \in \R^d,\ \ N \in \N,
 $$
 and
 $$
 \lim_{N \to \infty} \wt{a}_N(\z) = \wt{a}(\z),\quad \z \in \R^d.
 $$
By the Lebesgue theorem,
\begin{equation}
\label{5.4}
  \| \wt{a}_N - \wt{a} \|_{L_1(\Omega)} \to 0 \quad \text{as} \ N \to \infty.
  \end{equation}

For each $N\in \N$ we denote by $a_{N}(\z)$ the function $a_{N}(\z)=\wt a_{N}(\z)\cdot\mathbf{1}_{\Omega}(\z)$, $\z\in\mathbb{R}^{d}$.
Obviously, the function  $a_{N}$  possesses the following properties:
$$
a_{N}\in L_{1}(\R^{d}),\ \ a_{N}(\z)\ge 0,\ \ \z\in\R^{d},\ \ \|a_{N}\|_{L_{1}(\mathbb{R}^{d})}=\|\wt a_{N}\|_{L_{1}(\Omega)},
$$
$$
\wt a_{N}(\z)=\sum_{\n\in\mathbb{Z}^{d}}a_{N}(\z+\n).
$$
According to  the definition of $\wt{a}_N(\z)$ in \eqref{5.3}, the set of $\z\in\Omega$ for which  $\wt{a}_N(\z) >0$
coincides, for all $N\in\N$, with the set $\{\z\in\Omega\,:\, \wt{a}(\z) >0\}$; due to \eqref{5.2} the measure of this set is positive. Furthermore, by construction,
\begin{equation*}
\label{5.4a}
\wt{a}_N(\z) \le \wt{a}_{N+1}(\z), \quad \z \in \R^d, \quad N \in \N,
  \end{equation*}
and thus
\begin{equation*}
\label{5.5}
 0< \| \wt{a}_1 \|_{L_1(\Omega)} \le \| \wt{a}_N \|_{L_1(\Omega)} \le \| \wt{a} \|_{L_1(\Omega)}, \quad N \in \N.
  \end{equation*}
 We conclude that  $a_{N}$ satisfies all the conditions in  \eqref{h1.1}.

For each  $N \in \N$, we denote by  $\mathbf{A}_N$ a bounded linear operator in $L_2(\Omega)$
which is defined by the relations
\begin{equation*}
\begin{aligned}
\label{AN}
\mathbf{A}_N u (\x) &:= p_N(\x) u(\x) - \mathbf{B}_N u(\x),
\\
\mathbf{B}_Nu(\x)  &:= \intop_{\Omega}\widetilde a_N(\x-\y)\mu(\x,\y)u(\y)\,d\y,\ \ u\in L_{2}(\Omega),
\end{aligned}
\end{equation*}
with
\begin{equation*}
\label{5.6}
p_N(\x) = \intop_\Omega \widetilde a_N(\x-\y)\mu(\x,\y)\,d\y.
\end{equation*}
Due to \eqref{h1.2} and the definition of $p_N$, the following inequalities are valid:
\begin{equation*}
\label{5.6a}
0 <\mu_- \| \wt{a}_1\|_{L_1(\Omega)}\le \mu_- \| \wt{a}_N\|_{L_1(\Omega)} \le p_N(\x) \le \mu_+ \| \wt{a}_N\|_{L_1(\Omega)},
\end{equation*}
\begin{equation}\label{appr_p_N}
\| p_N - p\|_{L_\infty} \le \mu_+ \| \wt{a}_N - \wt{a}\|_{L_1(\Omega)}.
\end{equation}
By the Schur lemma,
\begin{equation}\label{f9}
\| \mathbf{B}_N - \mathbf{B}\|_{L_2(\Omega) \to L_2(\Omega)} \le \mu_+ \| \wt{a}_N - \wt{a}\|_{L_1(\Omega)}.
\end{equation}
From \eqref{appr_p_N} and \eqref{f9} it follows that
\begin{equation*}
\label{5.10}
\| \mathbf{A}_N - \mathbf{A}\|_{L_2(\Omega) \to L_2(\Omega)} \le 2 \mu_+ \| \wt{a}_N - \wt{a}\|_{L_1(\Omega)}.
\end{equation*}
Combining the last estimate with  \eqref{5.4} yields
$\| \mathbf{A}_N - \mathbf{A}\|_{L_2(\Omega) \to L_2(\Omega)} \to 0$, as $N \to \infty$.

Next, denoting by $\mathbf{G}_{N}$ the operator  $\mathbf{B}_{N}^{*}[p_{N}^{-1}]$ and taking into account \eqref{5.4}, \eqref{f9} and the inequality
$$
\|p_{N}^{-1}-p^{-1}\|_{L_{\infty}}\le(\mu_{-}\|\wt a_{1}\|_{L_{1}(\Omega)})^{-2}\mu_{+}\|\wt a_{N}-\wt a \|_{L_{1}(\Omega)},
$$
we obtain the relation
\begin{equation}
\label{5.10a}
\|\mathbf{G}_{N}-\mathbf{G}\|\to0,\quad  N\to+\infty.
\end{equation}

\subsection{Approximation for the Riesz projector  $\mathcal{P}$ and for the function $\psi_{0}$}
Since $\mathbf{G}_{N}^{*}\mathbf{1}_{\Omega}=\mathbf{1}_{\Omega}$, then $1\in\sigma(\mathbf{G}_{N})\cap\sigma(\mathbf{G}_{N}^{*})$.  Due to the compactness of the operators
$\mathbf{G}_{N}$ and $\mathbf{G}^{*}_{N}$,  there exists $r_{N}>0$ such that
$$
\sigma(\mathbf{G}_{N})\cap B_{2r_{N}}(1)=\{1\},\ \ \sigma(\mathbf{G}_{N}^{*})\cap B_{2r_{N}}(1)=\{1\}.
$$
Let us introduce the Riesz projectors of the operators $\mathbf{G}_{N}$ and $\mathbf{G}_{N}^{*}$ that correspond
to the point $\lambda=1$:
\begin{equation*}
\mathcal{P}_{N}=\frac{-1}{2\pi i}\oint_{|\zeta-1|=r_{N}}(\mathbf{G}_{N}-\zeta I)^{-1}\,d\zeta,\ \
\mathcal{P}_{N}^{*}=\frac{-1}{2\pi i}\oint_{|\zeta-1|=r_{N}}(\mathbf{G}_{N}^{*}-\zeta I)^{-1}\,d\zeta.
\end{equation*}
According to Proposition \ref{proposition2}, for each $N\in\N$, there exists a function   $\psi_{N}\in\operatorname{Ker}(\mathbf{G}_{N}-I)$  such that
\begin{equation*}\label{f10}
\psi_{N}(\z)\ge\gamma_{N}>0,\ \ \z\in\Omega,\ \ (\psi_{N},\mathbf{1}_{\Omega})=1,\ \
\operatorname{Ker}(\mathbf{G}_{N}-I)=\mathcal{L}\{\psi_{N}\},\ \ \mathcal{P}_{N}=(\cdot,\mathbf{1}_{\Omega})\psi_{N}.
\end{equation*}

\begin{lemma}\label{lemma6.4.1}
The following limit relations hold:
\begin{equation*}\label{fdop2}
\|\mathcal{P}_{N}-\mathcal{P}\|\to 0,\ \ \|\psi_{N}-\psi_{0}\|_{L_{2}(\Omega)}\to 0,\ \ \hbox{as }N\to+\infty.
\end{equation*}
\end{lemma}
\begin{proof}
Recalling that the point $\lambda=1$ is a simple isolated eigenvalue of the operator   $\mathbf{G}$,
we denote by $d_{0}$ the distance from the point  $1$ to the remaining part of the spectrum of operator $\mathbf{G}$.
According to Proposition \ref{rizprp3.2}, the annulus $\{\zeta\in\mathbb{C}:d_{0}/3\le |\zeta-1|\le 2d_{0}/3\}$
does not intersect the spectrum of the operator $\mathbf{G}_{N}$,  $N\ge N_{0}$, if $N_0$ is sufficiently large.
Also, for  $N>N_0$, the spectrum of $\mathbf{G}_{N}$ in the disc $\{\zeta\in\mathbb{C}:|\zeta-1|<d_{0}/3\}$
consists of one simple eigenvalue which is equal to $1$. Therefore,
\begin{equation}\label{f11}
\mathcal{P}_{N}-\mathcal{P}=\frac{-1}{2\pi i}\oint_{|\zeta-1|=d_{0}/2}\left((\mathbf{G}_{N}-\zeta I)^{-1}-(\mathbf{G}-\zeta I)^{-1}\right)d\zeta,\ \ N\ge N_{0}.
\end{equation}
Finally, from \eqref{rizprp3.2}, \eqref{5.10a} and \eqref{f11} we deduce that
\begin{equation*}\label{f12}
\|\mathcal{P}_{N}-\mathcal{P}\|\to 0\ \ \hbox{and }\ \psi_{N}=\mathcal{P}_{N}\psi_{0}\to\mathcal{P}\psi_{0}=\psi_{0}\ \ \text{in}\ \ L_{2}(\Omega),\ \ \hbox{as }N\to+\infty.
\end{equation*}
\end{proof}

\subsection{Uniform boundedness of the functions  $\psi_N $ in  $L_\infty$ norm.}
\label{Sec5.3}
Combining the identity
\begin{equation}\label{f13}
\psi_{N}(\x)=\mathbf{G}_{N}\psi_{N}(\x)=\int_{\Omega}\wt a_{N}(\y-\x)\mu(\y,\x)p_{N}^{-1}(\y) \psi_{N}(\y)\,d\y,
\end{equation}
with the relations $\wt a_{N}\le N$, $\mu\le\mu_{+}$, $p_{N}^{-1}\le(\mu_{-}\|\wt a_{1}\|_{L_{1}(\Omega)})^{-1}$, $\psi_{N}\ge 0$ and $(\psi_{N},\mathbf{1}_{\Omega})=1$,
we conclude that the function $\psi_{N}$ is bounded and, moreover,
$$
\|\psi_{N}\|_{L_{\infty}}\le N\mu_{+}(\mu_{-}\|\wt a_{1}\|_{L_{1}(\Omega)})^{-1}.
$$
Our goal is to obtain a uniform in $N$ estimate for the $L_\infty$ norm of $\psi_{N}$.

For the sake of brevity we denote
$$C_{0}:=\mu_{+}(\mu_{-}\|\wt a_{1}\|_{L_{1}(\Omega)})^{-1},\ \ \Omega_{N,1}:=\{\x\in\Omega:\psi_{N}(\x)>\|\psi_{N}\|_{L_{\infty}}^{1/2}\},\ \ \Omega_{N,2}:=\Omega\setminus\Omega_{N,1}.
$$
\begin{lemma}\label{lemma6.5.1}
The following estimate holds:
\begin{equation}\label{fdop3}
|\psi_{N}(\x)|\le C(\wt a, C_{0}),\ \ \x\in\Omega,\ \ N\in\N.
\end{equation}
\end{lemma}
\begin{proof}
Representing the integral in  \eqref{f13} as the sum of the integrals over $\Omega_{N,1}$ and $\Omega_{N,2}$ yields
\begin{multline}\label{f14}
\psi_{N}(\x)\le C_{0}\|\psi_{N}\|\big._{L_{\infty}}\int_{\Omega_{N,1}}\wt a_{N}(\y-\x)\,d\y+
C_{0}\|\psi_{N}\|_{L_{\infty}}^{1/2}\int_{\Omega_{N,2}}\wt a_{N}(\y-\x)\,d\y\le\\\le
C_{0}\|\psi_{N}\|\big._{L_{\infty}}\int_{\Omega_{N,1}}\wt a(\y-\x)\,d\y+
C_{0}\|\psi_{N}\|_{L_{\infty}}^{1/2}\int_{\Omega}\wt a(\y-\x)\,d\y,\ \ \x\in\Omega.
\end{multline}
The measure of the set $\Omega_{N,1}$ admits the estimate
\begin{equation}\label{f15}
\operatorname{mes}\Omega_{N,1}\le\int_{\Omega_{N,1}}\|\psi_{N}\|_{L_{\infty}}^{-1/2} \psi_{N}(\y)\,d\y\le\|\psi_{N}\|_{L_{\infty}}^{-1/2}\int_{\Omega}\psi_{N}(\y)\,d\y=\|\psi_{N}\|_{L_{\infty}}^{-1/2}.
\end{equation}
Denote
\begin{equation}
\label{5.22}
F_{\wt{a}}(t) := \sup \left\{  \int_{\mathcal O} \ \wt{a}(\z)\,d\z: \ {\mathcal O} \subset [-2,2]^d,\ \operatorname{mes} {\mathcal O} \le t \right\},\quad t>0.
\end{equation}
Obviously, the function $F_{\wt{a}}(t)$ is non-decreasing, and,  due to the properties of Lebesgue integral,
$F_{\wt{a}}(t) \to 0$, as $t\to +0$.
By  \eqref{f15} and \eqref{5.22} we have
$$
\intop_{\Omega_{N,1}}  \wt{a} (\y - \x)  \, d\y
\le F_{\wt{a}} \left(\|\psi_{N}\|_{L_{\infty}}^{-1/2}\right).
$$
Combined with  \eqref{f14} this estimate leads to the inequality
\begin{equation}
\label{5.23}
\| \psi_N \|_{L_\infty} \le C_{0}\|\psi_{N}\|_{L_{\infty}}F_{\wt{a}} \left(\|\psi_{N}\|_{L_{\infty}}^{-1/2}\right)+C_{0}\|\psi_{N}\|_{L_{\infty}}^{1/2}\|\wt a\|_{L_{1}(\Omega)}.
\end{equation}
Next, we select the number $t_0 = t_0(\wt{a},C_{0}) >0$ so that
$$
C_{0}\cdot F_{\wt a}(t_{0})\le\frac{1}{2}.
$$
If  $\|\psi_{N}\|_{L_{\infty}}^{-1/2} > t_0$, then
$$
\|\psi_{N}\|_{L_{\infty}}\le t_{0}^{-2}(\wt a, C_{0}).
$$
If $ \|\psi_{N}\|_{L_{\infty}}^{-1/2} \le t_0$, then
by \eqref{5.23} we have the inequality
$$
\|\psi_{N}\|_{L_{\infty}}\le\frac{1}{2}\|\psi_{N}\|_{L_{\infty}}+C_{0}\|\psi_{N}\|_{L_{\infty}}^{1/2}\|\wt a\|_{L_{1}(\Omega)},
$$
which implies that
$$
 \|\psi_{N}\|_{L_{\infty}}\le 4 C_{0}^{2}\|\wt a\|^{2}_{L_{1}(\Omega)}.
$$
Finally, we arrive at the estimate
\begin{equation*}
\|\psi_{N}\|_{L_{\infty}(\Omega)}\le\max\{t^{-2}_{0}(\wt a,C_{0}),4 C_{0}^{2}\|\wt a\|^{2}_{L_{1}(\Omega)}\}=:C(\wt a, C_{0}).
\end{equation*}
\end{proof}

\subsection{Boundedness of the function  $\psi_{0}$}
According to Lemma \ref{lemma6.4.1},
$$
\| \psi_N - \psi_{0} \|_{L_2(\Omega)} \to 0 \quad \text{as }\ N \to \infty.
$$
Then, by the Riesz theorem, for a subsequence  $N_k \to \infty$, it holds
$$
\psi_{N_k}(\x) \to \psi_{0}(\x), \quad \hbox{as }k \to \infty, \ \text{for almost all }\  \x \in \Omega.
$$
From this inequality, taking into account  \eqref{fdop3}, we obtain the estimate
$|\psi_{0}(\x)|\le C(\wt a,C_{0})$ for almost all $\x\in\Omega$.
This completes the proof of Theorem  \ref{theorem1.5}.

\end{document}